\documentclass[12pt,a4paper]{article}

\usepackage{latexsym}
\usepackage{graphicx}
\usepackage{subfigure}\usepackage{amsmath}
\usepackage{float}
\usepackage{color}
\usepackage{cite}
\usepackage{enumerate}
\usepackage{epsfig}
\usepackage{xcolor}
\usepackage{caption}
\usepackage{changes}

\newcommand{\proof}{\noindent{\bf Proof.}\ \ }
\newcommand{\qed}{\hfill $\Box$}

\newtheorem{theorem}{Theorem}
\newtheorem{claim}{Claim}
\newtheorem{lemma}{Lemma}

\newtheorem{observation}{Observation}

\begin{document}

\baselineskip=15pt
\parindent=0.5cm

\title{\Large {\bf  1-Planar graphs without $6$-cycles are 6-choosable }}

 \author{Qingqin Wu\\
 \normalsize School of Data Sciences, Zhejiang  University of Finance \& Economics, Hangzhou\\
 \normalsize 310018, China
 \and
 Yiqiao Wang\thanks{Corresponding author. Email: yqwang@bjut.edu.cn;
 Research supported by NSFC (No.\,12422113).}\\
 \normalsize School of Mathematics, Statistics and Mechanics, Beijing University of Technology,\\
 \normalsize Beijing 100124, China}
 \maketitle

\maketitle

%
\begin{abstract}

\baselineskip=13 pt
A graph is 1-planar if it can be drawn in the plane so that each edge is crossed by at most one other edge.
A graph  is  $k$-degenerate if each of its subgraphs contains a vertex of degree at most $k$.
It was known that every 1-planar graph is 8-choosable.
In this paper, we show that every 1-planar graph without 6-cycles is 5-degenerate and hence 6-choosable.
\bigskip


\end{abstract}
%

\baselineskip=17 pt

\section{Introduction}

Only simple graphs are considered in this paper. Given a graph $G$, let $V(G)$, $E(G)$, $\Delta(G)$ and $\delta(G)$ denote its vertex set, edge set, maximum degree and minimum degree, respectively. For a vertex $v\in V(G)$, let $d_G(v)$ denote the degree of $v$ in $G$.
We say that $v$ is a {\em $k$-vertex}, {\em $k^+$-vertex}, or {\em $k^-$-vertex} if $d_G(v)=k$, $d_G(v)\geq k$, or $d_G(v)\leq k$.
Let $C_k$ denote a cycle of length $k$ in $G$.
A {\em trail} is a walk that traverses each edge at most once.
Call a set $S$ {\em full} if no two elements in $S$ are the same.
For $X\subseteq E(G)$, we use $G[X]$ to denote the subgraph of $G$ induced by $X$.
For two nonnegative integers $p,q$, we use $[p,q]$ to denote the set of all
integers between $p$ and $q$ (including $p$ and $q$).

A proper coloring of a graph $G$ is a mapping $\phi$ from $V(G)$ to the color set $\{1,2,\ldots,k\}$ such that
$\phi(x)\ne \phi(y)$ for every pair of adjacent vertices $x$ and $y$ of $G$.  We say that $L$
is a list assignment for the graph $G$ if it assigns a list $L(v)$ of possible colors to
each vertex $v$ of $G$. If $G$ has a proper coloring $\phi$ such that $\phi(v) \in L(v)$ for all vertices $v$, then we say that
$G$ is {\em $L$-colorable}.  The graph $G$ is $k$-{\em choosable} if it is $L$-colorable for every assignment $L$ satisfying
$|L(v)|\ge k$  for every vertex $v$. A graph $G$ is $k$-{\em degenerate} if every subgraph $H$ of $G$ contains a vertex of degree at most $k$ in $H$. Obviously, every $d$-degenerate graph is $(d+1)$-choosable.

A {\em planar graph} is a graph which can be embedded in the plane such that any two  edges intersect only at their ends.
Such a drawing is called a {\em plane graph}. Thomassen \cite{thom} showed that every planar graph is 5-choosable, whereas Voigt \cite{voi} constructed a planar graph on 238 vertices that is not 4-choosable.
Since a planar graph without 3-cycles is 3-degenerate,  it is 4-choosable.
Moreover, it was shown that every planar graph without $k$-cycles is 4-choosable for $k = 4$ in \cite{lam},  for $k = 5$ in \cite{wang2}, for $k = 6$ in \cite{fij}, and  for $k = 7$ in \cite{far}.

A {\em $1$-planar graph} is a graph that can be drawn in the plane such that each edge crosses at most one other edge.
A {\em $1$-plane graph} means such a drawing of a 1-planar graph. Fabrici and Madaras \cite{fa} showed that every 1-planar graph $G$ satisfies $|E(G)|\le 4|V(G)|-8$, and this bound is attainable. We say that $G$ is {\em optimal} if $|E(G)|=4|V(G)|-8$,
and {\em IC-planar} if every vertex is incident with at most one crossing edge.

Borodin \cite{boro} proved that 1-planar graphs are 6-colorable, and
Kr\'{a}${\rm l}'$ and Stacho \cite{kral} proved that IC-planar graphs are 5-colorable. Both results are the best possible.
Since 1-planar graphs are 7-degenerate, they are 8-choosable. A result in \cite{wl} implies that
optimal 1-planar graphs are 7-choosable. Yang et al.\cite{yang} showed that IC-planar graphs are 6-choosable.
Dvo\v{r}\'{a}k,   Lidick\'{y} and \v{S}krekovski \cite{dvo1} showed that every graph with at most two crossings is 5-choosable.
A strong result, due to Dvo\v{r}\'{a}k et al.\cite{dvo2}, asserts that every IC-planar graph with  the distance between every pair of crossings at least 15 is 5-choosable.

It is conjectured in \cite{yang} that every 1-planar graph is 6-choosable.
If this conjecture were true, then it would strengthen the Borodin's result in \cite{boro}.
Fabrici and Madaras \cite{fa} showed that every  1-planar graph without 3-cycles is 5-degenerate.
Hud\'{a}k and Madaras \cite{hud}  showed that every 1-planar graph without 4-cycles is $5$-degenerate.
Wu, Wang and Kong \cite{wu3} showed that every 1-planar graph without 5-cycles is $5$-degenerate.
These results imply immediately the following:

\begin{theorem}\label{thm1}  For a fixed $k\in [3,5]$, every $1$-planar graph without $k$-cycles
is $6$-choosable.
\end{theorem}

In this paper, we are going to extend this result to the case of $k=6$, i.e., we will prove the following:

\begin{theorem}\label{thm2} Every $1$-planar graph without $6$-cycles is $6$-choosable.
\end{theorem}

Since any subgraph of a 1-planar graph without 6-cycles is a 1-planar graph without 6-cycles,
instead of showing Theorem \ref{thm2}, we prove the following stronger result:

\begin{theorem}\label{thm3}
Every $1$-planar graph without $6$-cycles has $\delta(G)\le 5$.
\end{theorem}

\section{Preliminary}

Let $G$ be a plane graph. Let $F(G)$ denote the face set of $G$.
For a face $f\in F(G)$, let $d_G(f)$ denote the degree of $f$ in $G$.
A face $f\in F(G)$ is called a {\em $k$-face}, {\em $k^+$-face}, or {\em $k^-$-face} if $d_G(f)=k$, $d_G(f)\ge k$, or $d_G(f)\le k$.
Usually, we use $\partial(f)$ to denote the boundary walk of $f$ and write $f=[u_1u_2\cdots u_k]$ if $u_1, u_2, \ldots, u_k$ are the vertices of $\partial(f)$ in cyclic order. Set $V(f)=V(\partial(f))$ and $E(f)=E(\partial(f))$.
For a face $f$ and an edge $e\in E(f)$, let $f_e$   denote the face adjacent to $f$ such that $e\in E(f)\cap E(f_e)$.

Let $H$ be a 1-plane graph such that the number of crossings is as few as possible.
Assume that $E_0(H)$ is the set of non-crossed edges in $H$. Let $X(H)$ denote the set of crossings in $H$. The {\em associated plane graph}, denoted $H^{\times}$, of $H$ is a plane graph with
$V(H^{\times})=V(H)\cup X(H)$ and $E(H^{\times})=E_0(H)\cup E_1(H)$,  where

\medskip

$E_1(H)=\{xz, zy \ |\ xy\in E(G'')\setminus E_0(H)$ and $z$ is a crossing  on $xy\}$.

\medskip

Vertices in $V(H)$ are called {\em true vertices} of $H^{\times}$, and vertices in $X(H)$ are called {\em false vertices} of $H^{\times}$.
Note that $d_{H^{\times}}(v)=d_H(v)$ for each  $v\in V(H)$, and $d_{H^{\times}}(v)=4$ for each  $v\in X(H)$.
A face is {\em false} if it is incident with a false vertex and {\em true} otherwise.
Since $H$ is 1-plane,  the following statement holds obviously:

\medskip
\noindent{\bf (P1)}\ Two false vertices are not adjacent in $H^\times$.
\medskip

For $x\in V(H^{\times})\cup F(H^{\times})$ and an integer $i\ge3$, we use $m_i(x)$ and $m_{i^+}(x)$ denote the number of $i$-faces and $i^+$-faces that are adjacent to or incident with $x$.

\section{Proof}

We say that a cycle $C$ is {\em true} if all edges of $C$ are not crossing edges.
For a true cycle $C$, we use $V_{\rm int}(C)$ and $V_{\rm ext}(C)$ to denote the set of vertices inside $C$ and outside $C$, respectively. Similarly, let $E_{\rm int}(C)$ and $E_{\rm ext}(C)$ denote the set of all edges contained inside and outside $C$, respectively. Set $E_{\rm Int}(C)=E_{\rm int}(C)\cup E(C)$, and $G_{\rm Int}(C)=G[E_{\rm Int}(C)]$.
We say that $C$ is {\em separating} if both $V_{\rm int}(C)\ne \emptyset$ and $V_{\rm ext}(C)\ne \emptyset$.

\medskip

\noindent{\bf Proof of Theorem \ref{thm3}.} Suppose that the theorem is false.
Let $G$ be a connected counterexample, which is drawn in the plane  such that the number of crossings is as few as possible.
Then $G$ is a 1-plane graph without $6$-cycles and $\delta(G)\ge 6$.
To complete the proof, we need to carry out the following operations according to their order.
\medskip

\noindent
{\bf Step 1.}\  If $G$ is 2-connected, then let $G'=G$.
Otherwise, we choose $G'$ as an end block

\ \ \ \ \ \ \ \
of $G$ which contains only one cut vertex $z^*$ of $G$.

\medskip

\noindent
{\bf Step 2.}\  If $G'$ has separating true  $4^-$-cycles, then we choose a separating
true  $4^-$-cycle

\ \ \ \ \ \ \ \ \ \
 $C^*$ with $|V_{\rm int}(C^*)|+|E_{\rm int}(C^*)|$ being as small as possible,
and let $G''=G'_{\rm Int}(C^*)$;

\ \ \ \ \ \ \ \
Otherwise,  let $G''=G'$.

\smallskip

Then $G''$ is a 2-connected 1-plane graph satisfying
  (Q1)--(Q2) below:
\medskip

\noindent
{\bf (Q1)}\ $G''$ contains no $6$-cycles;
\smallskip

\noindent
{\bf (Q2)}\ $G''$ contains no  separating true $4^-$-cycles.

\medskip

The following Claim \ref{3path} is easily derived from (Q2):

\begin{claim}\label{3path}
Assume that $P=b_0b_1\cdots b_m$ is a trail of $G''$ with $3\leq m\leq4$ and all edges of $P$ are not crossing edges.
For $i\in [1,m-1]$,   let the neighbors of $b_i$ be $b_{i-1},y_{i_1},\ldots,y_{i_{p_i}},b_{i+1},$ $z_{i_1},\ldots,z_{i_{q_i}}$ in clockwise order,
as shown in {\rm Fig.\,\ref{fig1}}.
Define
$$Y=(\mathop\cup\limits_{i=1}^{m-1}\{y_{i_1},\ldots,y_{i_{p_i}}\})\backslash V(P),$$
$$Z=(\mathop\cup\limits_{i=1}^{m-1}\{z_{i_1},\ldots,z_{i_{q_i}}\})\backslash V(P).$$
If $|Y|,|Z|\geq1$, then $b_0\ne b_m$.
\end{claim}

    \begin{figure}[H]
 	\centering
 \epsfig{file=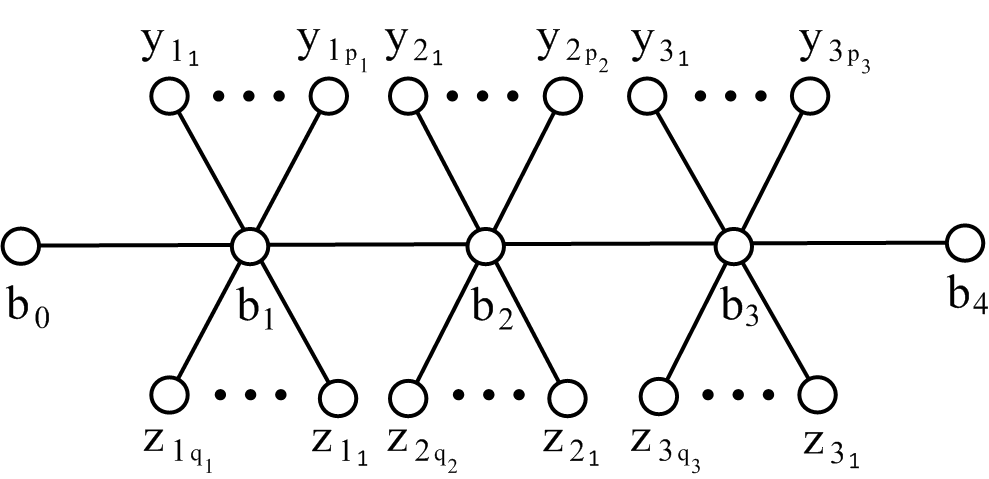, width=9cm}
 	\caption{ A trail $P=b_0b_1b_2b_3b_4$ described  in Claim \ref{3path}.}
 	\label{fig1}
 \end{figure}

In this and following figures, we  use  $\otimes$, $\bullet$,  and $\circ$   to denote a false vertex, a true vertex of known degree, and a vertex of unknown degree,  respectively.

Let $H$  denote the associated plane graph of $G''$.
Let $V^o=V(C^*)$ if $C^*$ exists, and $V^o=\{z^*\}$ if
$z^*$ exists, but $C^*$ does not exist. Let $V^i=V(H)\setminus V^o$.
Vertices in $V^i$ and vertices in $V^o$ are called {\em internal vertices}
and {\em outer vertices} of $H$, respectively.
Let $f^*$ denote the {\em outer face} of $H$ if $C^*$ exists, and let $F^i=F(H)
\setminus \{f^*\}$. The faces in $F^i$ are called {\em internal
faces} of $H$.
It is easy to see that $H$ is a 2-connected plane graph with the following properties (Q3)--(Q4):

\medskip
\noindent
{\bf (Q3)}\ If $C^*$ exists, then $3\leq d_{H}(f^*)\leq 4$ and $d_{H}(v)\geq2$ for each $v\in V(C^*)$.
If $C^*$ does

\ \ \ \ \,not exist, then $d_{H}(z^*)\ge2$.

\smallskip
\noindent
{\bf (Q4)}\ $d_{H}(v) = d_G(v)\geq6$  for a true internal vertex $v$, and $d_{H}(v)=4$ for a false vertex $v$.

\medskip

Let $w$ denote an initial  weight function defined by $w(x)=d_{H}(x)-4$ for
every $v\in V(H)\cup F(H)$.  Then
the total sum of weights is equal to $-8$. We are going to design some
discharging rules and redistribute weights according to them. Once
the discharging is finished, a new weight function $w'$  is
produced. However, the total sum of weights is kept fixed when the
discharging is in process.
On the other hand, we can show that
$w'(x)\geq0$ for all $x\in V^i\cup F^i$ and
$\sum\limits_{x\in V^o}w'(x)+w'(f^*)\geq-\frac{239}{30}$,
which  leads to an obvious contradiction below:
$$
-\frac{239}{30}\leq \sum_{x\in V(H)\cup F(H)}w'(x) = \sum_{x\in V(H)\cup F(H)}w(x)=-8,
$$
and henceforth the proof is complete.

\medskip

For $x,y\in V(H)\cup F(H)$, we  use $\tau(x \to y)$ to denote the sum
of weights discharged from $x$ to $y$ according to the
following rules.
\medskip

{\bf (R0)}\ Let $v\in V^o$ and $f$ be an internal 3-face incident with $v$.
\smallskip

{\bf (R0.1)}\ If $V^o=\{z^*\}$, then $\tau(z^*\rightarrow f)=1$.
\smallskip

{\bf (R0.2)}\ Suppose that $V^o=V(C^*)$. Then $\tau(v\rightarrow f)=\frac{1}{2}$ if $f$ is incident with two outer vertices and $\tau(v\rightarrow f)=\frac{14}{15}$ otherwise.
\smallskip

{\bf (R0.3)}\ Let $g$ be an internal $4^+$-face  which is  adjacent to $f$ and incident with $v$. Then $v$ sends additional $\frac{7}{15}$ to $f$ through $g$.
\smallskip

{\bf (R1)}\ Every internal $6^+$-vertex  $v$ divides equally    $d_{H}(v)-4$ to each incident 3-face.
\smallskip

{\bf (R2)}\ Every $5^+$-face $f$ divides equally $d_{H}(f)-4$ to each adjacent false 3-face

\medskip

For an internal 3-face $f$, let $\sigma(f)$ denote the resultant weight of $f$ after (R0)--(R2) are carried out.
We say that  $f$ is  {\em saturated}  if $\sigma(f)\ge 0$.
Otherwise, $f$ is {\em unsaturated}.

Assume that $v$ is a false vertex of $H$.
Let $T_v$ denote the set of  false 3-faces incident with $v$ in $H$.
The weight of  $T_v$ is defined to be the sum of weights of elements in $T_v$.
Since no two false vertices are adjacent by (P1), every false 3-face belongs to only one $T_v$ for some
false vertex $v$.
Similarly, let $\sigma(T_v)$ denote the resultant weight of $T_v$ after (R0)--(R2) are implemented.
Then we define that $T_v$ is  {\em saturated}  if $\sigma(T_v)\ge 0$. Otherwise, $T_v$ is {\em unsaturated}.

For two false 3-face $f\in T_v$  and $g\in T_u$, we call that $T_v$ and $T_u$ are adjacent if $f$ is adjacent to $g$ in $H$.
Moreover,  we use $\Lambda(T_v)$ to denote the set of $T_u$ which are unsaturated and adjacent to $T_v$.  Let $\lambda(T_v)=|\Lambda(T_v)|$. Obviously, $0\leq\lambda(T_v)\leq4$.

\medskip
Then we carry out the following additional rules:
\medskip

{\bf (r1)}\ For a true 3-face $f\in F^i$ with $\sigma(f)>0$, we devide equally $\sigma(f)$ to each adjacent false 3-face.
\smallskip

{\bf (r2)}\
If $\sigma(T_v)>0$, then we divide equally $\sigma(T_v)$ to the elements in $\Lambda(T_v)$.
\medskip

Let $w'$ denote the  final weight function of $H$ after (R0)--(R2) and  (r1)--(r2) are performed.
From now on, we write simply $d(x)=d_{H}(x)$ for every $x\in V(H)\cup F(H)$.
The following Observation \ref{observation} holds automatically.

\begin{observation}\label{observation}
No $2$-vertex in $V(C^*)$ is incident with an internal $3$-face.
An internal $3$-face  is  incident with at most two outer vertices.
\end{observation}

\begin{lemma}\label{C*}
$\eta:=\sum\limits_{x\in V^o}w'(x)+w'(f^*)\geq-\frac{239}{30}$.
\end{lemma}

\proof
If $V^o=\{z^*\}$, then $\eta=w'(z^*)=d(z^*)-4-m_3(z^*)-2\times\frac{7}{15}m_{4^+}(z^*)\ge d(z^*)-4-d(z^*)=-4$ by (R0.1) and (R0.3).
Otherwise, $V^o=V(C^*)$ and $f^*$ exists. Let $C^*=x_0x_1\cdots x_{p-1}x_0$, where $p\in\{3,4\}$.
By (Q3), $3\le d(f^*)\le 4$ and $d(x_i)\geq2$ for  $i\in[0,p-1]$.

Let $i\in [0,p-1]$. If $d(x_i)=2$, then $w'(x_i)=w(x_i)=2-4=-2$ by Observation \ref{observation}.
If $d(x_i)\geq3$, then it is easy to see that both $(f^*)_{x_{i-1}x_i}$ and
$(f^*)_{x_{i}x_{i+1}}$ cannot be internal 3-faces which are incident with only one outer vertex.
Let $g\in\{(f^*)_{x_{i-1}x_i}, (f^*)_{x_{i}x_{i+1}}\}$.
In fact, by (R0.2)--(R0.3), $x_i$ sends $\frac{1}{2}$ to $g$ if $d(g)=3$ and sends $\frac{7}{15}$ at most once through $g$ if $d(g)\ge4$. If $d(x_i)=3$, then $w'(x_i)\ge3-4-\frac{1}{2}\times2=-2$. Otherwise, $d(x_i)\ge4$.
Let $h$ be an internal face, other than $(f^*)_{x_{i-1}x_i}$ and $(f^*)_{x_{i}x_{i+1}}$, which is incident with $x_i$.
According to (R0.2)--(R0.3), it is easy to see that $x_i$ sends at most $\frac{14}{15}$ to $h$ if $d(h)=3$ and sends $\frac{7}{15}$ at most twice through $h$ if $d(h)\ge4$.
Hence, $w'(x_i)\geq d(x_i)-4-\frac{1}{2}\times2-\frac{14}{15}(d(x_i)-3)
=\frac{1}{15}d(x_i)-\frac{11}{5}\ge-\frac{29}{15}$.

If $d(f^*)=3$, then $w'(f^*)=w(f^*)=3-4=-1$. So it follows that $\eta\geq(-2)\times3+(-1)=-7$.
Otherwise, $d(f^*)=4$ and   $w'(f^*)=w(f^*)=4-4=0$. If $d(x_i)\ge 4$ for some $i\in [0,3]$,
then $\eta\geq(-2)\times3+(-\frac{29}{15})=-\frac{119}{15}$.
Otherwise, $2\le d(x_i)\le 3$ for all $i\in[0,3]$.
Since $H$  is 2-connected, $V(C^*)$ contains at least two $3$-vertices, say  $d(x_0)=3$ by symmetry.
If $m_3(x_0)\leq1$, then $w'(x_0)\leq 3-4-\frac{1}{2}-\frac{7}{15}=-\frac{59}{30}$ by (R0.2)--(R0.3).
Therefore, $\eta\geq(-2)\times3+(-\frac{59}{30})=-\frac{239}{30}$.
Otherwise, $m_3(x_0)=2$.
Observation \ref{observation} implies that $x_1$ and $x_3$
are $3$-vertices and hence $m_3(x_1)=m_3(x_3)=2$, similarly to the above proof.
Consequently,  $d(x_i)=3$ and $m_3(x_i)=2$ for all $i\in[0,3]$.
This means that $G''$ is a complete graph $K_4$ or a wheel graph $W_4$, contradicting  the choice of $C^*$ and (Q4).
\qed

\begin{claim}\label{6+v} Let $v$ be an internal $6^+$-vertex  incident  with a $3$-face $f$.

$(1)$\ If $d(v)=6$, then $\tau(v\to f)\ge \frac 13$. Moreover, $\tau(v\to f)\ge \frac 2{5}$ if $m_{4^+}(v)\ge 1$,  $\tau(v\to f)\ge \frac 1{2}$ if $m_{4^+}(v)\ge 2$ and $\tau(v\to f)\ge \frac 2{3}$ if $m_{4^+}(v)\ge 3$.

$(2)$\ If $d(v)\ge7$, then $\tau(v\to f)\ge \frac 37$. Moreover, $\tau(v\to f)\ge \frac 1{2}$ if $m_{4^+}(v)\ge 1$, $\tau(v\to f)\ge \frac 3{5}$ if $m_{4^+}(v)\ge 2$ and $\tau(v\to f)\ge \frac 3{4}$ if $m_{4^+}(v)\ge 3$.
\end{claim}

\proof
Note that $d(v)\ge6$ and  $m_3(v)\le d(v)$.

$(1)$\ Suppose that $d(v)=6$. By (R1), we have that $\tau(v\rightarrow f)=\frac{d(v)-4}{m_3(v)}\ge\frac{d(v)-4}{d(v)}\ge\frac{1}{3}$.
If $m_{4^+}(v)\ge 1$, then $\tau(v\rightarrow f)\ge\frac{d(v)-4}{d(v)-1}\ge\frac{2}{5}$ by (R1).
If $m_{4^+}(v)\ge 2$, then $\tau(v\rightarrow f)\ge\frac{d(v)-4}{d(v)-2}\ge\frac{1}{2}$ by (R1).
If $m_{4^+}(v)\ge 3$, then $\tau(v\rightarrow f)\ge\frac{d(v)-4}{d(v)-3}\ge\frac{2}{3}$ by (R1).

$(2)$\ Similar to the proof of $(1)$.
\qed

\begin{lemma}\label{internal-vertex}
If $v\in V^0$, then $w'(v)\geq0$.
\end{lemma}
\proof   If $v$ is   false, then $d(v)=4$ by (Q4) and hence $w'(v)=w(v)=4-4=0$.
 Otherwise,  $d(v)\geq6$ by (Q4).  (R1)  implies that  $w'(v)\geq0$.\qed

\begin{lemma}\label{internal-4-face}
If $f\in F^i$ is not a false $3$-face, then $w'(f)\ge 0$.
\end{lemma}
\proof  Since $H$ is a simple graph,  $d(f)\ge 3$.
If $d(f)=4$, then $w'(f)=w(f)=4-4=0$.
If $d(f)\geq5$, then (R2) implies  that  $w'(f)\ge 0$.
Otherwise, $f$ is a true 3-face.
Then each of the vertices incident with $f$ is
an outer vertex or an internal $6^+$-vertex. By (R0) and Claim \ref{6+v}, $w'(f)\ge3-4+3\times\frac{1}{3}=0$.\qed
\smallskip

Now it remains to prove the following conclusion:

\begin{lemma}\label{false-3-face}
If $f\in F^i$ is a  $3$-face incident with a false vertex $v$, then $w'(f)\ge 0$.
\end{lemma}

We introduce some notation that will be used throughout the proofs below.
Fix such a false vertex $v$, let $x,y,z,w$ be the neighbors of $v$ in $H$ in cyclic order.
For $i\in[0,3]$, let $f_i$ denote the incident face  of $v$ in $H$ such that $vx,vy\in E(f_0), vy,vz\in E(f_1), vz,vw\in E(f_2), vw,vx\in E(f_3)$. Set $U=\{x,y,z,w\}$. For $u\in U$, let $v,u_1,u_2,\cdots,u_{d(u)-1}$ be the neighbors of $u$ in $H$ in cyclic order.
For $i\in[1,d(u)-2]$, let $f_u^i$ denote the incident face of $u$ in $H$ such that $uu_i,uu_{i+1}\in E(f_u^i)$.
For $i\in [1,4]$, let $X_i$ denote the set of false vertices in $H$ which are incident with exactly $i$ 3-faces.
Define $X_2=X^{n}_2\cup X^a_2$, where $X^n_2$ and $X^a_2$ are the sets of vertices in $X_2$ which are incident with two non-adjacent and adjacent 3-faces, respectively.

Instead of proving Lemma~\ref{false-3-face} directly, we prove the following four theorems, which together imply the lemma.

\begin{theorem}\label{X_1}
If $f$ is a false $3$-face incident with a false vertex $v\in X_1\cup X_2^n$, then either $w'(f)\ge 0$ or $G''$ contains a $6$-cycle.
\end{theorem}
\proof
Assume, w.l.o.g., that  $f=f_0$ with $d(f_0)=3$, $d(f_2)\ge3$ and $d(f_1), d(f_3)\ge4$.
Note that $w(f)=w(f_0)=3-4=-1$ and $x=y_1$,  $y=x_{d(x)-1}$.
Let $u\in\{x,y\}$.  It is easy to see  that $m_{4^+}(u)\ge1$.
By Claim \ref{6+v} and (R0), we obtain that $\tau(u\rightarrow f)\ge\frac{2}{5}$.
If $x\in V^o$, then $\tau(x\rightarrow f)\ge\frac{1}{2}$ by (R0.1)--(R0.2) and $x$ sends additional $\frac{7}{15}$ to $f$ through $f_3$ by (R0.3).
So, $w'(f)\ge-1+\frac{2}{5}+\frac{1}{2}+\frac{7}{15}=\frac{11}{30}$.
Hence, $x,y\in V^i$ and $d(x),d(y)\ge6$ by (Q4) and symmetry.
If $d(f_i)\ge5$ for some $i\in\{1,3\}$, then  $\tau(f_i\rightarrow f)\ge\frac{d(f_i)-4}{d(f_i)}\ge\frac{1}{5}$ by (R2) and therefore
$w'(f)\ge-1+\frac{2}{5}\times2+\frac{1}{5}=0$.
Thus, $d(f_1)=d(f_3)=4$.
Suppose that $x_1, y_{d(y)-1}$ are true vertices, then let  $S=U\cup\{y_{d(y)-1}\}$. It is easy to check that no two vertices in $S$ are identical, and $d_H(x_1, u)\le 2$ for each $u\in U$.
Moreover, since $d(x)\ge6$ by (Q4), $P=x_1xyy_{d(y)-1}$ is a trail satisfying the conditions of Claim  \ref{3path}.
By Claim  \ref{3path},  $x_1\ne y_{d(y)-1}$.
This implies that $S\cup\{x_1\}$ is full.
We set $C_6=xx_1wyy_{d(y)-1}zx$, which is a 6-cycle.
Then the proof is split into two cases as follows  by symmetry.

\medskip
\noindent{\bf Case 1.} $y_{d(y)-1}$ is  true  and $x_1$ is  false.
\medskip

If $d(f_x^1)=3$, then $x_2\in V(G'')$ and we find a 6-cycle $C_6=xx_2wyy_{d(y)-1}zx$,  similar to the above discussion.
Hence, $d(f_x^1)\ge4$ and therefore $\tau(x\rightarrow f)\ge\frac{1}{2}$ by Claim \ref{6+v}.
Next, we claim that $d(y)=6$ and $m_{4^+}(y)=1$.
Otherwise, either $d(y)\ge7$ and $m_{4^+}(y)\ge1$, or  $d(y)=6$ and $m_{4^+}(y)\ge2$.
In either subcase, Claim \ref{6+v} implies $\tau(y\rightarrow f)\ge\frac{1}{2}$.
Consequently, $w'(f)\ge-1+\frac{1}{2}\times2=0$.
Thus $f_y^1=[xy_2y]$. If $y_2$ is true, then at least one of $y_3$ and $y_4$ is true by (P1).
Define $y^*=y_4$ if $y_4$ is true, otherwise $y^*=y_3$.
Let $S=\{x,y,z,y_2,y^*,y_5\}$.
Clearly $S$ is  full, and we obtain a 6-cycle $C_6=xzy_5y^*yy_2x$, see thick lines in Fig.\,\ref{fig2}.
Otherwise, $y_2$ is false and $y_3$ is true. If $y_4$ is true,  then $C_6=xzy_5y_4yy_3x$ is a 6-cycle.
Now suppose $y_4$ is false.
Then we claim that $d(x)=6$ and $m_{4^+}(x)=2$;
otherwise, by Claim \ref{6+v},  $\tau(x\rightarrow f)\ge\frac{3}{5}$,
and hence $w'(f)\ge-1+\frac{3}{5}+\frac{2}{5}=0$.
So, $x_3$ is true and $xx_3\in E(G'')$. We obtain a 6-cycle $C_6=xzy_5y_3yx_3x$.

\medskip
\noindent{\bf Case 2.} $y_{d(y)-1}$ and $x_1$ are false.
\medskip

If $\tau(u\rightarrow f)\ge\frac{1}{2}$ for each $u\in\{x,y\}$, then $w'(f)\ge-1+\frac{1}{2}\times2=0$.
By Claim \ref{6+v} and symmetry, we may assume that $d(y)=6$ and $m_{4^+}(y)=1$.
Then $y_4$ is true.
If $d(f_x^1)=3$, then as before we obtain a $C_6=xx_2wyy_4zx$, where $x_2\ne y_4$ by Claim \ref{3path} and  $S=U\cup\{x_2,y_4\}$  is full.
Hence, $d(f_x^1)\ge4$.
Next we claim that $d(x)=6$ and $m_{4^+}(x)=2$; otherwise, by Claim \ref{6+v}, $w'(f)\ge-1+\frac{3}{5}+\frac{2}{5}=0$.
Now we consider the following two situations.
If $y_2=x_4$ is true, then we define $x^*=x_3$ if $x_3$ is true, otherwise $x^*=x_2$.
Let  $S=\{x,y,z,y_2,x^*\}$. No two vertices in $S$ are identical, and $d_H(y_4, u)\le 2$ for each $u\in S\setminus\{x^*\}$.
Moreover, $P=y_4yxx^*$ is a trail satisfying the conditions of Claim  \ref{3path} with $z$ and $y_2$ belonging to $Y$ and $Z$, respectively.
By Claim  \ref{3path},  $y_4\ne x^*$, so $S\cup\{y_4\}$ is full.
Then we  find a  $C_6=xzy_4yy_2x^*x$.
If $y_2=x_4$ is false, then $x_3,y_3\in V(G'')$ and $y_4y_3\in E(G'')$.
Obviously, $S=\{x,y,z,x_3,y_3,y_4\}$ is full and we set $C_6=xzy_4y_3yx_3x$.
\qed

    \begin{figure}[H]
 	\centering
 \epsfig{file=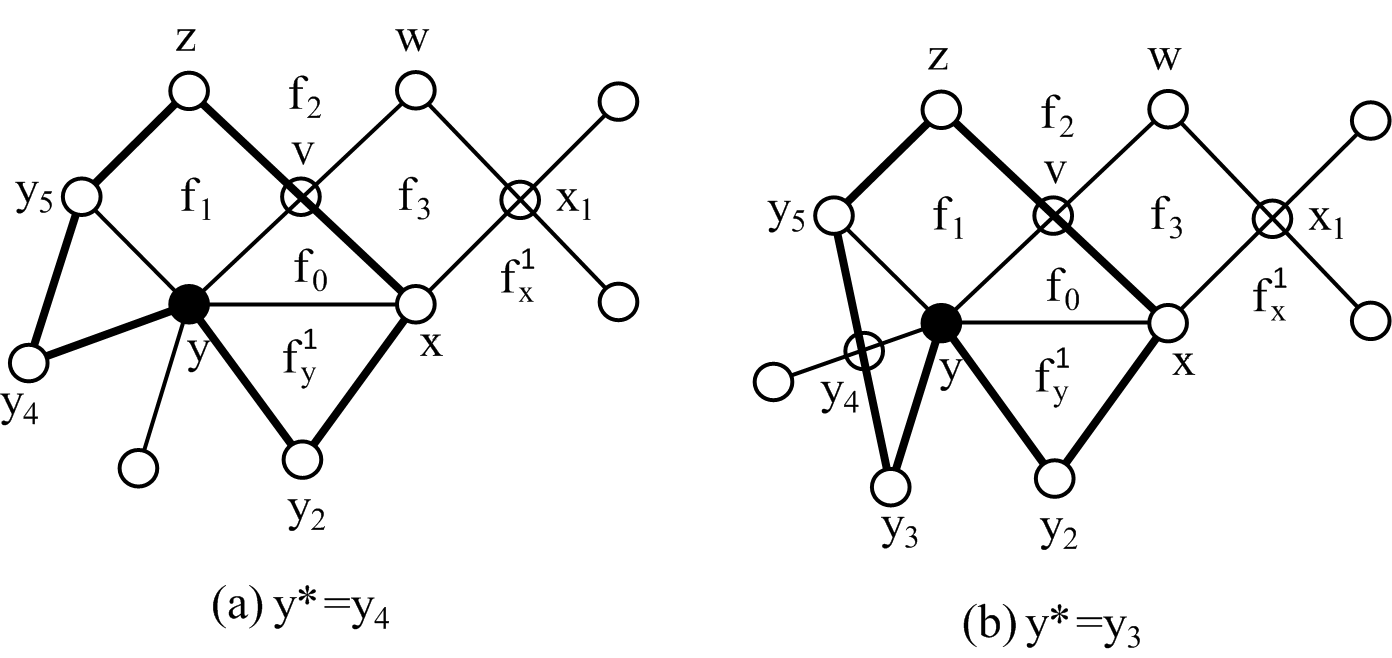, width=12.5cm}
 	\caption{ Two configurations in the proof of Theorem \ref{X_1}.}
 	\label{fig2}
 \end{figure}

In Theorems \ref{X_2}--\ref{X_4}, we do not directly show $w'(f)\ge0$ for a false $3$-face $f$; instead we show that for the incident false vertex $v$, either $w'(T_v)\ge0$ or $G''$ contains a $6$-cycle.

\begin{theorem}\label{X_2}
If $v\in X_2^a$ is a false vertex, then either $w'(T_v)\ge0$ or $G''$ contains a $6$-cycle.
\end{theorem}

\begin{theorem}\label{X_3}
If $v\in X_3$ is a false vertex, then either $w'(T_v)\ge0$ or $G''$ contains a $6$-cycle.
\end{theorem}

\begin{theorem}\label{X_4}
If $v\in X_4$ is a false vertex, then either $w'(T_v)\ge0$ or $G''$ contains a $6$-cycle.
\end{theorem}

\section{Proof of Theorem \ref{X_2}}
We assume, w.l.o.g., that $d(f_0)=d(f_3)=3$ and $d(f_1), d(f_2)\ge4$.
Note that $w(T_v)=w(f_0)+w(f_3)=(3-4)\times2=-2$.
In this case, $x=y_1$, $y=x_{d(x)-1}$ and $w=x_1$.
By Claim \ref{6+v} and (R0), $\tau(u\rightarrow T_v)\ge \frac{2}{5}$ for $u\in\{y,w\}$ and $\tau(x\rightarrow T_v)\ge \frac{1}{3}\times2=\frac{2}{3}$.
If $y\in V^o$, then $\tau(y\rightarrow f_0)\ge\frac{1}{2}+\frac{7}{15}=\frac{29}{30}$ by (R0).
Hence, $w'(T_v)\ge-2+\frac{29}{30}+\frac{2}{5}+\frac{2}{3}=\frac{1}{30}$.
So, $y,w\in V^i$  by symmetry.
If $x\in V^o$, then $\tau(x\rightarrow T_v)\ge\frac{14}{15}\times2=\frac{28}{15}$ by (R0).
Therefore, we have that $w'(T_v)\ge-2+\frac{28}{15}+\frac{2}{5}\times2=\frac{2}{3}$.
So, $x\in V^i$.

Suppose that $d(f_i)\ge5$ for $i=1,2$, then $\tau(f_i\rightarrow T_v)\ge\frac{d(f_i)-4}{d(f_i)-1}\ge\frac{1}{4}$  by (R2).
We first claim that $d(x)=6$ and $m_{4^+}(x)=0$; otherwise $\tau(x\rightarrow T_v)\ge \frac{2}{5}\times2=\frac{4}{5}$ by Claim \ref{6+v}, and then $w'(T_v)\ge-2+\frac{4}{5}+\frac{2}{5}\times2+\frac{1}{4}\times2=\frac{1}{10}$.
Next we claim that $d(y)=6$ and $m_{4^+}(y)=1$; otherwise $\tau(y\rightarrow T_v)\ge\frac{1}{2}$, which would give $w'(T_v)\ge-2+\frac{1}{2}+\frac{2}{5}+\frac{2}{3}+\frac{1}{4}\times2=\frac{1}{15}$.
By symmetry, $d(w)=6$ and $m_{4^+}(w)=1$.
If $y_2=x_4$ is a true vertex, then define $x^*=x_3$ if $x_3$ is true and $x^*=x_2$ otherwise; similarly, define $y^*=y_3$ if $y_3$ is true and $y^*=y_4$ otherwise.
It is easy to check that  $S=\{x,y,w,x_4,x^*,y^*\}$ is full and   we set $C_6=xwyy^*x_4x^*x$.
Thus, $y_2=x_4$ and $x_2=w_4$ are false  by symmetry.
By (P1), $x_3, y_3, w_3$ are true.
Obviously, $S=\{x,y,w,x_3,y_3,w_3\}$ is full and we set $C_6=xy_3yx_3ww_3x$.
Hence, either $f_1$ or $f_2$ is a 4-face. Then the proof is split into two cases as follows  by symmetry.

\medskip
\noindent{\bf Case 1.} $d(f_1)\ge5$ and $d(f_2)=4$.
\medskip

Note that $f_2=[wvzw_1]$ and $\tau(f_1\rightarrow T_v)\ge\frac{d(f_1)-4}{d(f_1)-1}\ge\frac{1}{4}$ by (R2).
Suppose that $w_1$ is a true vertex, then $S=U\cup\{w_1\}$ is full.
If $d(f_y^1)=3$, then $f_y^1=[xy_2y]$.
When $y_2$ is a true vertex, we have $d_H(y_2, u)\le 2$ for each $u\in U$.
Moreover, since $d(w)\ge6$ by (Q4), $P=y_2xww_1$ is a trail satisfying the conditions of Claim \ref{3path}.
By Claim \ref{3path}, $y_2\ne w_1$, which implies that $S\cup\{y_2\}$ is full.
We then set $C_6=xzw_1wyy_2x$.
When $y_2$ is false, we claim that $d(f_y^1)_{xy_2}, d(f_y^1)_{yy_2}\ge4$.
Otherwise, $P_1=x_{d(x)-3}xww_1$ and $P_2=y_3yxww_1$ satisfy the conditions of Claim \ref{3path}.
Clearly, $S\cup\{x_{d(x)-3}\}$ and $S\cup\{y_3\}$ are full, and we set $C_6=xzw_1wyx_{d(x)-3}x$ and $C_6=xzw_1wyy_3x$.
Consequently, $m_{4^+}(x)\ge1$ and $m_{4^+}(y)\ge2$.
By Claim \ref{6+v}, $\tau(x\rightarrow T_v)\ge\frac{2}{5}\times2=\frac{4}{5}$ and $\tau(y\rightarrow T_v)\ge\frac{1}{2}$.
Then we claim that $d(w)=6$ and $m_{4^+}(w)=1$.
Otherwise, $\tau(w\rightarrow T_v)\ge\frac{1}{2}$ and then $w'(T_v)\ge-2+\frac{1}{2}\times2+\frac{4}{5}+\frac{1}{4}=\frac{1}{20}$ by Claim \ref{6+v}. We define  $w^*=w_2$ if $w_2$ is true, and $w^*=w_3$  otherwise.
It is easy to check that $S\cup\{w^*\}$ is full, and so $C_6=xzw_1w^*wyx$ is a 6-cycle.

Suppose that $w_1$ is a false vertex.
If $d(f_w^1)=3$, then the proof can be given similarly to the above discussion.
Hence  $d(f_w^1)\ge4$. By Claim \ref{6+v}, $\tau(w\rightarrow T_v)\ge\frac{1}{2}$.

\begin{claim}\label{xyz}
At least two of the following statements hold:
$(1)$\  $d(x)=6$ and $m_{4^+}(x)=0$;\
$(2)$\  $d(y)=6$ and $m_{4^+}(y)=1$;\
$(3)$\  $d(w)=6$ and $m_{4^+}(w)=2$.
\end{claim}





\proof Suppose that at most one of  (1)--(3) holds.  We  consider three cases  as follows.
If (1) and (2) are false, then  $\tau(x\rightarrow T_v)\ge\frac{2}{5}\times2=\frac{4}{5}$ and  $\tau(y\rightarrow T_v)\ge\frac{1}{2}$ by Claim \ref{6+v}.
Hence $w'(T_v)\ge-2+\frac{4}{5}+\frac{1}{2}+\frac{1}{2}+\frac{1}{4}=\frac{1}{20}$.
If (2) and (3) are false, then  $\tau(y\rightarrow T_v)\ge\frac{1}{2}$ and $\tau(w\rightarrow T_v)\ge\frac{3}{5}$ by Claim \ref{6+v}.
Therefore $w'(T_v)\ge-2+\frac{2}{3}+\frac{1}{2}+\frac{3}{5}+\frac{1}{4}=\frac{1}{60}$.
If (1) and (3) are false, then  $\tau(x\rightarrow T_v)\ge\frac{2}{5}\times2=\frac{4}{5}$ and  $\tau(w\rightarrow T_v)\ge\frac{3}{5}$ by Claim \ref{6+v}.
Thus $w'(T_v)\ge-2+\frac{4}{5}+\frac{2}{5}+\frac{3}{5}+\frac{1}{4}=\frac{1}{20}$.
\qed

\medskip

Now we have to  discuss the following  two subcases.

\medskip
\noindent{\bf Case 1.1.} (2) and (3) are true.
\medskip

First, define $w^*=w_4$ if $w_4$ is true, otherwise $w^*=w_3$. Set $S=\{x,y,w,w^*\}$, which is full.
If $y_2$ is true, define $y^*=y_3$ if $y_3$ is true, otherwise $y^*=y_4$. Taking $P=y^*yxww^*$ with $z\in Y$ and $y_2\in Z$, Claim \ref{3path} gives $y^*\ne w^*$; thus $S\cup\{y_2,y^*\}$ is full and we set $C_6=xy_2y^*yww^*x$.
If $y_2$ is false, then $y_3$ is true. Define $y^*=y_4$ if $y_4$ is true,  otherwise $y^*=y_5$. By the same argument as above, $S\cup\{y_3,y^*\}$ is full and we set $C_6=xy_3y^*yww^*x$.

\medskip
\noindent{\bf Case 1.2.} (1) is true.
\medskip

This implies that exactly one of (2) and (3) holds. If $x_i$ are true vertices for all $i\in[2,4]$, then we have $C_6=xwx_2x_3x_4yx$.
Now we consider the remaining possibilities, depending on which of $x_2,x_3,x_4$ are false.

$\bullet$ If $x_2$ is false while $x_3,x_4$ are true, then two subcases arise.
If (2) holds, define $y^*=y_3$ if $y_3$ is true and $y^*=y_4$ otherwise.
Then $S=\{x,y,w,x_3,x_4,y^*\}$ is full and we set $C_6=xwx_3x_4y^*yx$.
If (3) holds, then $w_3$ is true and $G''$ contains $C_6=xx_3x_4yww_3x$.
The case where $x_4$ is false while $x_2,x_3$ are true is symmetric and is handled similarly.

$\bullet$ If $x_3$ is false while $x_2,x_4$ are true, we may assume without loss of generality that (2) holds.
Define $y^*=y_3$ if $y_3$ is true and $y^*=y_4$ otherwise.
Then $S=\{x,y,w,x_2,x_4,y^*\}$ is full and $G''$ contains $C_6=xx_4y^*ywx_2x$.

$\bullet$ If $x_2$ and $x_4$ are false while $x_3$ is true, then two subcases occur.
If (2) holds, then $y_3$ is true. Define $y^*=y_4$ if $y_4$ is true,  otherwise $y^*=y_5$.
Then $S=\{x,y,w,x_3,y_3,y^*\}$ is full and we set $C_6=xx_3wyy^*y_3x$.
Otherwise (3) holds, so $w_3$ is true. The above discussion implies $d(f_y^2)\ge4$.
By Claim \ref{6+v}, $\tau(y\rightarrow T_v)\ge \frac{1}{2}$.
If  $\tau(f_1\rightarrow T_v)\ge \frac{1}{3}$,  then $w'(T_v)\ge -2+\frac{2}{3}+\frac{1}{2}\times2+\frac{1}{3}=0$, and we are done.
This implies $d(f_1)=5$ and $m_{4^+}(f_1)=1$ by (R2).
Write $f_1=[vyy_5z_1z]$.
If $y_5$ and $z_1$ are true, then Claim \ref{3path} applied to $P=z_1y_5yxw$ gives $z_1\ne w$, so $U\cup\{y_5,z_1\}$ is full and we set $C_6=xwyy_5z_1zx$.
If $y_5$ is true and $z_1$ is false, then $d((f_1)_{z_1y_5})=3$.
Let $(f_1)_{z_1y_5}=[z_1y_5s]$ with a true vertex $s$,  then $U\cup\{y_5,s\}$ is full, yielding $C_6=xwyy_5szx$ similarly.
If $y_5$ is false and $z_1$ is true, then $y_4$ is true. If $U\cup\{y_4,z_1\}$ is full, we set $C_6=xwyy_4z_1zx$, see Fig.\,\ref{fig3}(a). Otherwise $z_1=w$, and then $S=\{x,y,w,w_3,x_3,y_4\}$ is full, giving $C_6=ww_3xx_3yy_4w$, see Fig.\,\ref{fig3}(b).

    \begin{figure}[H]
 	\centering
 \epsfig{file=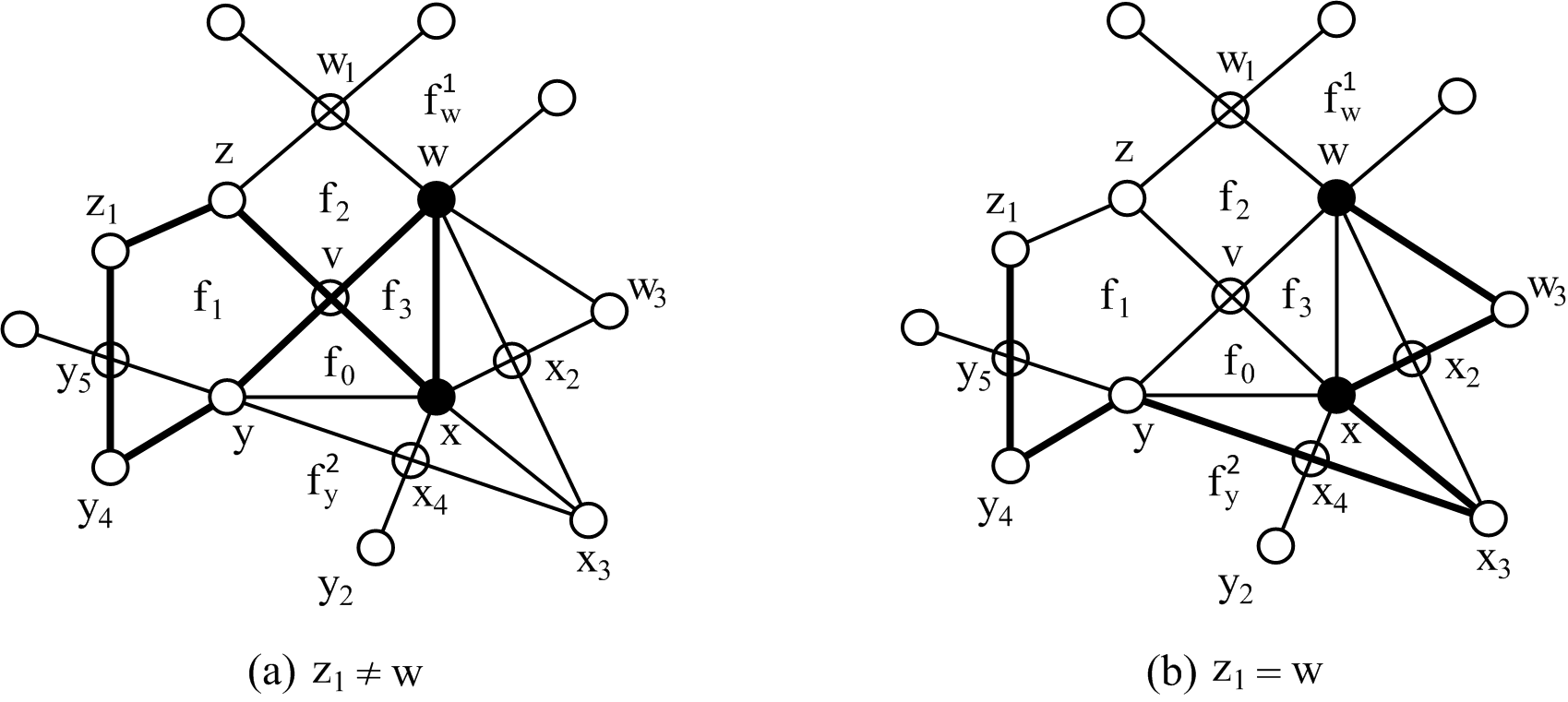, width=12.5cm}
 	\caption{ Two configurations in the proof of Case 1.2.}
 	\label{fig3}
 \end{figure}

\medskip
\noindent{\bf Case 2.} $d(f_1)=d(f_2)=4$.
\medskip

Let $f_1=[vyy_{d(y)-1}z]$ and $f_2=[vzw_1w]$.
If $w_1$ and $y_{d(y)-1}$ are true vertices, then
it is easy to check that $U\cup\{w_1,y_{d(y)-1}\}$ is full, and we set $C_6=xyy_{d(y)-1}zw_1wx$.
By symmetry, we  consider  following two subcases.

\medskip
\noindent\textbf{Case 2.1.} $w_1$ is false and $y_{d(y)-1}$ is true.
\medskip

Note that $S=U\cup\{y_{d(y)-1}\}$ is full.
If $d(f_w^1)=3$, then $w_2$ is true and we set $C_6=xyy_{d(y)-1}zw_2wx$.
Thus $d(f_w^1)\ge4$.
Now consider $f_x^1$.
Suppose $d(f_x^1)=3$, i.e., $f_x^1=[xwx_2]$.
If $x_2$ is true, then Claim \ref{3path} applied to $P=x_2xyy_{d(y)-1}$ gives $x_2\ne y_{d(y)-1}$; hence $S\cup\{x_2\}$ is full and we set $C_6=xzy_{d(y)-1}ywx_2x$.
If $x_2$ is false, then $(f_x^1)_{x_2x}$ and $(f_x^1)_{x_2w}$ are $4^+$-faces, otherwise $G''$ contains a $6$-cycle by Claim \ref{3path}.
Consequently $m_{4^+}(x)\ge1$ and $m_{4^+}(w)\ge3$.
By Claim \ref{6+v}, $\tau(x\to T_v)\ge\frac{2}{5}\times2=\frac{4}{5}$ and $\tau(w\to T_v)\ge\frac{2}{3}$.

Suppose $d(f_y^{d(y)-3})=d(f_y^{d(y)-2})=3$, define $y^*=y_{d(y)-2}$ if $y_{d(y)-2}$ is true and $y^*=y_{d(y)-3}$ otherwise. Then $S\cup\{y^*\}$ is full, and we set $C_6=xzy_{d(y)-1}y^*ywx$.
Hence $d(f_y^{d(y)-3})\ge4$ or $d(f_y^{d(y)-2})\ge4$, so $m_{4^+}(y)\ge2$.
We claim that $d(y)=6$ and $m_{4^+}(y)=2$; otherwise $\tau(y\to T_v)\ge\frac{3}{5}$ by Claim \ref{6+v}, and then $w'(T_v)\ge-2+\frac{4}{5}+\frac{3}{5}+\frac{2}{3}=\frac{1}{15}$.
Thus $d(f_y^1)=d(f_y^2)=3$ and $\tau(y\to T_v)=\frac{1}{2}$ by Claim \ref{6+v}.
If $y_2$ is true, then $f_y^1$ is an internal true $3$-face with
$\sigma(f_y^1)\ge 3-4+\frac{1}{2}+\frac{2}{5}+\frac{1}{3}=\frac{7}{30}$ after (R0)--(R2) are carried out.
By (r1), $\tau(f_y^1\to f_0)\ge\frac{7}{30}\times\frac{1}{3}=\frac{7}{90}$.
Consequently $w'(T_v)\ge-2+\frac{4}{5}+\frac{1}{2}+\frac{2}{3}+\frac{7}{90}=\frac{2}{45}$.
If $y_2$ is false, then $y_3$ is true. We claim that $d(x)=6$ and $m_{4^+}(x)=1$; otherwise $\tau(x\to T_v)\ge\frac{1}{2}\times2=1$ by Claim \ref{6+v}, and then $w'(T_v)\ge-2+1+\frac{1}{2}+\frac{2}{3}=\frac{1}{6}$.
Thus $x_3$ is true.
We now consider two possibilities.

$\bullet$ Suppose $f_x^1$ is a $4^+$-face.
If $d(f_x^1)\ge5$, then $\tau(f_x^1\to f_3)\ge\frac{d(f)-4}{d(f)}\ge\frac{1}{5}$ by (R2),  and hence $w'(T_v)\ge-2+\frac{4}{5}+\frac{1}{2}+\frac{2}{3}+\frac{1}{5}=\frac{1}{6}$.
Otherwise $d(f_x^1)=4$.
If $x_2$ is true, then Claim \ref{3path} applied to $P=x_2xyy_5$ gives $x_2\ne y_5$, so $S'=\{x,y,z,x_2,x_3,y_5\}$ is full and we set $C_6=xx_2x_3yy_5zx$.
If $x_2$ is false, then $f_x^1=[xww_{d(w)-2}x_2]$. By (P1), $w_{d(w)-2}$ is true. Claim \ref{3path} applied to $P=w_{d(w)-2}wxyy_3$ gives $w_{d(w)-2}\ne y_3$, so $S'=\{x,y,w,x_3,y_3,w_{d(w)-2}\}$ is full and we set $C_6=xy_3yww_{d(w)-2}x_3x$.

$\bullet$ Suppose $f_x^1$ is a $3$-face. Then by the previous discussion $d(f_x^2)\ge4$.
Let $x,w,s_1,s_2$ be the neighbors of $x_2$ and let $g=(f_x^2)_{x_2s_2}$.
If $d(g)=3$, then $S'=\{x,y,w,x_3,s_1,s_2\}$ is  full obviously and we set $C_6=xx_3yws_2s_1x$.
Hence $d(g)\ge4$, so $x_2\in X_1$.
It is not difficult to see that $\lambda(T_{x_2})\le1$ and $\sigma(T_{x_2})\ge-1+\frac{2}{5}+\frac{2}{3}=\frac{1}{15}$ after (R0)--(R2) are implemented.
By (r2), $\tau(T_{x_2}\to T_v)\ge\frac{1}{15}$.
Consequently $w'(T_v)\ge-2+\frac{4}{5}+\frac{1}{2}+\frac{2}{3}+\frac{1}{15}=\frac{1}{30}$.

\medskip
\noindent\noindent{\bf Case 2.2.} $w_1$  and $y_{d(y)-1}$ are  false.
\medskip

If $d(f_w^1)=3$ or $d(f_y^{d(y)-2})=3$, then the proof is similar to that in Case 2.1.
Hence we may assume $d(f_w^1)\ge4$ and $d(f_y^{d(y)-2})\ge4$.
Note that by Claim \ref{6+v}, $\tau(u\rightarrow T_v)\ge\frac{1}{2}$ for $u\in\{y,w\}$.
If $d(x)\ge8$, or $d(x)=7$ with $m_{4^+}(x)\ge1$, or $d(x)=6$ with $m_{4^+}(x)\ge2$, then $\tau(x\to T_v)\ge\frac{1}{2}\times2=1$ by (R1).
Therefore, $w'(T_v)\ge-2+1+\frac{1}{2}\times2=0$.

If $d(x)=7$ with $m_{4^+}(x)=0$, then $\tau(x\to T_v)\ge\frac{3}{7}\times2=\frac{6}{7}$ by Claim \ref{6+v}.
If in addition  $m_{4^+}(y)\ge3$, then $w'(T_v)\ge-2+\frac{6}{7}+\frac{2}{3}+\frac{1}{2}=\frac{1}{42}$ by Claim \ref{6+v}.
Hence, $m_{4^+}(y)=m_{4^+}(w)=2$.
Suppose that  $x_2$ and $x_5$ are true vertices, then we define $x^*=x_3$ if $x_3$ is true, otherwise $x^*=x_4$ .
Obviously, $S=\{x,y,w,x_2,x_5,x^*\}$ is full and we set $C_6=xx_5ywx_2x^*x$.
Suppose that  $x_2$ is a true vertex and $x_5$ is a false vertex, then $y_3$ is true and we define $x^*$ in the same way.
It is easy to check that $S=\{x,y,w,x_2,y_3,x^*\}$ is full and we set $C_6=xy_3ywx_2x^*x$.
Suppose that  $x_2$ and $x_5$ are false vertices,
then $x_3, x_4, y_3$ are true vertices by (P1).
Obviously, $S=\{x,y,w,x_3,x_4,y_3\}$ is full and we set $C_6=xx_4x_3wyy_3x$.
Now it remains to consider the case $d(x)=6$, which falls into two subcases.

\medskip
\noindent{\bf Case 2.2.1.} $d(x)=6$ and $m_{4^+}(x)=1$.
\medskip

Suppose that $d(f_x^1)\ge4$, then $\tau(x\rightarrow T_v)\ge \frac{2}{5}\times2=\frac{4}{5}$ and $\tau(w\rightarrow T_v)\ge \frac{2}{3}$ by Claim \ref{6+v}.
If $d(f_x^1)\ge5$, then $\tau(f_x^1\rightarrow f_3)\ge \frac{1}{5}$ by (R2) and hence $w'(T_v)\ge-2+\frac{4}{5}+\frac{1}{2}+\frac{2}{3}+\frac{1}{5}=\frac{1}{6}$.
Thus $d(f_x^1)=4$, i.e., $f_x^1=[xww_{d(w)-2}x_2]$.
By (P1), at most one  of $x_2$ and $w_{d(w)-2}$ is false.
Define $x^*=x_4$ if $x_4$ is true, and $x^*=x_3$ otherwise.
We claim that   $m_{4^+}(w)=3$;
otherwise $\tau(w\rightarrow T_v)\ge\frac{d(w)-4}{d(w)-4}=1$ by (R1) and then $w'(T_v)\ge-2+\frac{4}{5}+\frac{1}{2}+1=\frac{3}{10}$.
If $x_2$ and $w_{d(w)-2}$ are true, then Claim \ref{3path} asserts  that $w_{d(w)-2}\ne x^*$ by taking $P=w_{d(w)-2}wxx^*$ and hence $S=\{x,y,w,x_2,x^*,w_{d(w)-2}\}$ is  full.
We set $C_6=xx^*yww_{d(w)-2}x_2x$.
If $x_2$ is true and $w_{d(w)-2}$ is false, then $w_{d(w)-3}$ is true.
Then $S=\{x,y,w,x_2,x^*,w_{d(w)-3}\}$ is full, and we set $C_6=xx^*yww_{d(w)-3}x_2x$.
If $x_2$ is false and $w_{d(w)-2}$ is true, then $x_3$ is true.
Define $w^*=w_{d(w)-3}$ if $w_{d(w)-3}$ is true, and $w^*=w_{d(w)-4}$ otherwise.
Then $S=\{x,y,w,x_3,w_{d(w)-2},w^*\}$ is full,
and we set $C_6=xyww^*w_{d(w)-2}x_3x$.

Next,  by symmetry we consider the case where $d(f_x^2)\ge4$ and $d(f_x^i)=3$ for $i\in\{1,3,4\}$.
For some $u\in\{x,y\}$. If  $u_2$ is true,  then $f_u^1$ is a true 3-face with $\sigma(f_u^1)\ge \frac{1}{2}+\frac{2}{5}+\frac{1}{3}=\frac{7}{30}$ after (R0)--(R2) are implemented. By (r1), $\tau(f_u^1\rightarrow T_v)\ge\frac{7}{30}\times\frac{1}{3}=\frac{7}{90}$.
Then we claim that $m_{4^+}(y)=m_{4^+}(w)=2$; otherwise  $w'(T_v)\ge-2+\frac{4}{5}+\frac{1}{2}+\frac{2}{3}+\frac{7}{90}=\frac{2}{45}$ by Claim \ref{6+v}.
Suppose that $x_2$ is true, then $f_x^1$ is a true 3-face.
Define $x^*=x_4$ if $x_4$ is true, otherwise $x^*=x_3$;  define $w^*=w_{d(w)-3}$ if $w_{d(w)-3}$ is true, otherwise $w^*=w_{d(w)-4}$.
By taking $P=w^*wxx^*$, Claim  \ref{3path} implies that $w^*\ne x^*$, so $S=\{x,y,w,x_2,w^*,x^*\}$ is full, and we set $C_6=xx^*yww^*x_2x$.
Hence, $x_2$ is false.
Suppose that $y_2$ is true, then $f_y^1$ is a true 3-face.
According to the above discussion, $w_{d(w)-3}$ is  true.
Define $y^*=y_3$ if $y_3$ is true, otherwise $y^*=y_4$.
By taking $P=y^*yxww_{d(w)-3}$, Claim  \ref{3path}
implies that $y^*\ne w_{d(w)-3}$ and hence $S=\{x,y,w,x_4,y^*,w_{d(w)-3}\}$ is full.
It suffices to set $C_6=xx_4y^*yww_{d(w)-3}x$.

Now suppose that $x_2$ and  $x_4$ are both false.
By (P1), $x_3$ is true.
Let $x,w,s_1,s_2$ be the neighbors of $x_2$ in $H$ and let $g=(f_x^2)_{x_2s_2}$.
Since $d(f_x^2)\ge4$, we have $s_2\ne x_3$.
If $d(g)=3$, then  $C_6=xx_3yws_2s_1x$ is a 6-cycle; thus $d(g)\ge4$.
If $d(f_w^{d(w)-3})\ge4$, then $x_2\in X_1$.
As in Case 3.2.1, we obtain that $\tau(T_{x_2}\rightarrow T_v)\ge\frac{1}{15}$ and hence  $w'(T_v)\ge-2+\frac{4}{5}+\frac{1}{2}+\frac{2}{3}+\frac{1}{15}=\frac{1}{30}$ by Claim \ref{6+v}  and (r2).
Thus $d(f_w^{d(w)-3})=3$.
If $d(f_x^2)=4$, then  $C_6=xyx_3s_2ws_1x$  is a 6-cycle.
If $d(f_x^2)=5$, write $f_x^2=[xx_2s_2s_3x_3x]$.
Suppose that $s_3$ is true, then $S=\{x,y,w,x_3,s_2,s_3\}$ is full  and we set $C_6=xyx_3s_3s_2wx$.
Otherwise, $s_3$ is a false vertex.
Similarly, we have that $(f_x^2)_{x_3s_3}$ is a $4^+$-face.
Thus, $m_{4^+}(f_x^2)\ge2$  and then $\tau(f_x^2\rightarrow f_x^1)\ge\frac{5-4}{5-2}=\frac{1}{3}$ by (R2).
If $d(f_x^2)\ge6$, then $\tau(f_x^2\rightarrow f_x^1)\ge\frac{d(f_x^2)-4}{d(f_x^2)-1}=\frac{2}{5}$ by (R2).
Consequently, $d(f_x^2)\ge5$ and $\tau(f_x^2\rightarrow f_x^1)\ge\frac{1}{3}$.

    \begin{figure}[H]
 	\centering
 \epsfig{file=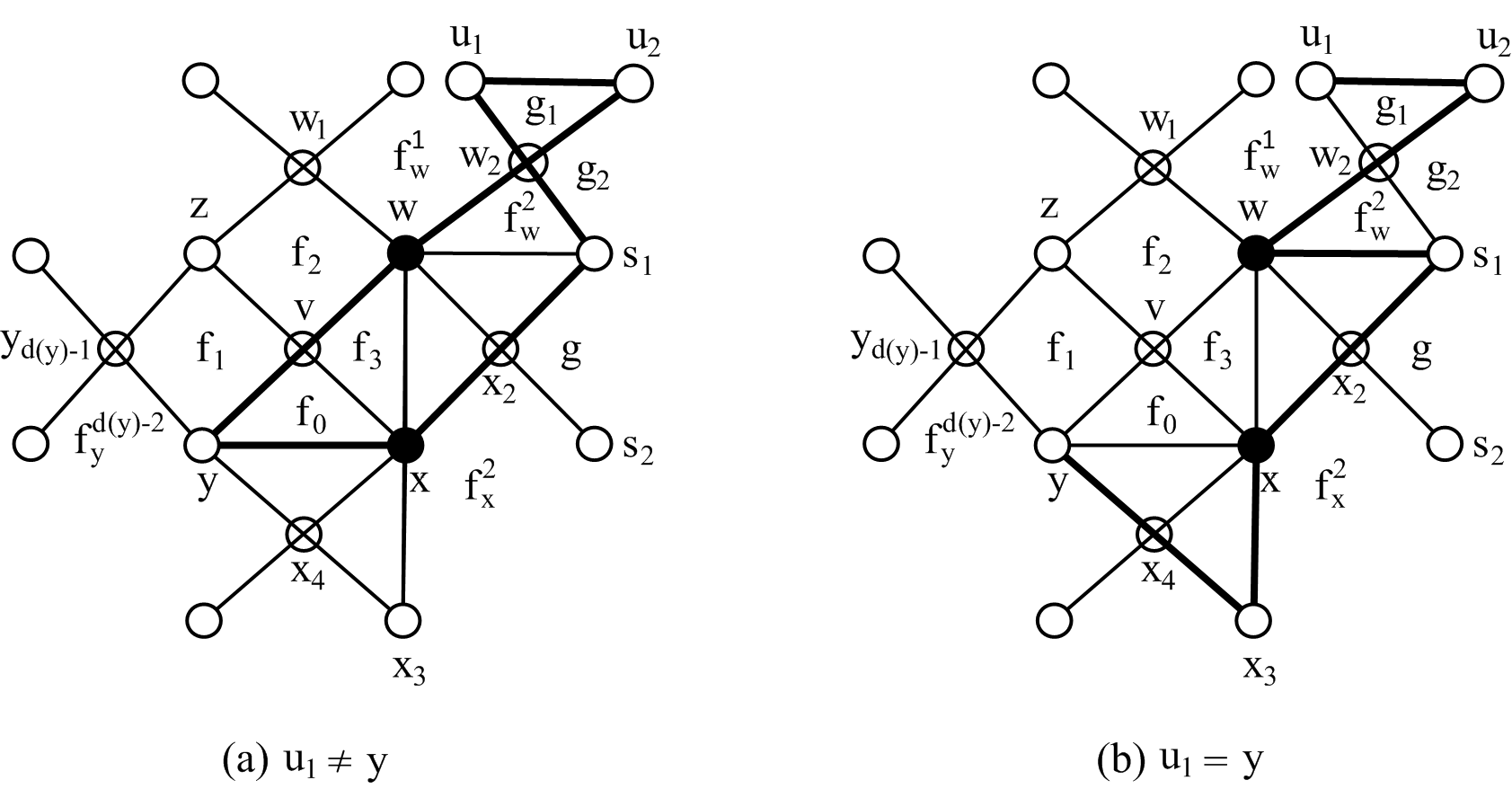, width=12.5cm}
 	\caption{ Two configurations in the proof of Case 2.2.1.}
 	\label{fig4}
 \end{figure}

We claim that $d(w)=6$ and $m_{4^+}(w)=2$.
Otherwise, $\tau(w\rightarrow T_v)\ge\frac{3}{5}$ and $\tau(w\rightarrow T_{x_2})\ge\frac{3}{5}\times2=\frac{6}{5}$ by Claim \ref{6+v}. By (R0)--(R1),   $\tau(s_1\rightarrow T_{x_2})\ge\frac{2}{5}$.
Thus after (R0)--(R2) are carried out,   $\sigma(T_{x_2})\ge -2+\frac{6}{5}+\frac{2}{5}\times2+\frac{1}{3}=\frac{1}{3}$, so $T_{x_2}$ is saturated.
Since $\lambda(T_{x_2})\le2$,  (r2) gives $\tau(T_{x_2}\rightarrow T_v)\ge\frac{1}{3}\times\frac{1}{2}=\frac{1}{6}$, and then $w'(T_v)\ge-2+\frac{4}{5}+\frac{1}{2}+\frac{3}{5}+\frac{1}{6}=\frac{1}{15}$.
If $w_2$ is true, then Claim \ref{3path} applied to  $P=w_2s_1wxx_3$ yields $w_2\ne x_3$,  so $S=\{x,y,w,x_3,s_1,w_2\}$ is full and $C_6=xx_3yww_2s_1x$  is a $6$-cycle.
Hence $w_2$ is false.
Let  $s_1,w,u_1,u_2$ be the neighbors  of   $w_2$, and let $f_w^2,f_w^1,g_1,g_2$ be its  incident faces in $H$.
If $d(g_2)=3$, then we  set $C_6=xx_3ywu_2s_1x$.
If $d(g_1)=3$, then $S=\{x,y,w,s_1,u_1,u_2\}$ is full  when $u_1\ne y$ and  we set $C_6=xywu_2u_1s_1x$, see Fig.\,\ref{fig4}(a).
When $u_1=y$, $S=\{x,y,w,x_3,s_1,u_2\}$ is full and  we set $C_6=yu_2ws_1xx_3y$, see Fig.\,\ref{fig4}(b).
Thus $d(g_1)\ge4$, $d(g_2)\ge4$,  so $w_2\in X_1$.
By (R0)--(R1), $\tau(s_1\rightarrow T_{x_2})\ge \frac{1}{2}$ and $\tau(s_1\rightarrow T_{w_2})\ge \frac{1}{2}$.
After (R0)--(R2) are implemented, $\sigma(T_{w_2})\ge-1+\frac{1}{2}\times2=0$ and $\sigma(T_{x_2})\ge-2+\frac{1}{2}\times3+\frac{2}{5}+\frac{1}{3}=\frac{7}{30}$.
Therefore, $\lambda(T_{x_2})=1$ and $\tau(T_{x_2}\rightarrow T_v)\ge\frac{7}{30}$ by (r2).
Consequently, $w'(T_v)\ge-2+\frac{4}{5}+\frac{1}{2}\times2+\frac{7}{30}=\frac{1}{30}$.

\medskip
\noindent{\bf Case 2.2.2.} $d(x)=6$ and $m_{4^+}(x)=0$.
\medskip

If $m_{4^+}(y),m_{4^+}(w)\ge3$, then $w'(T_v)\ge-2+\frac{2}{3}\times3=0$.
Hence, w.l.o.g.,  assume that $m_{4^+}(y)=2$.
Suppose  $x_4$ is  true.
Define $y^*=y_3$ if $y_3$ is true, otherwise $y^*=y_4$; define $x^*=x_3$ if $x_3$ is true, otherwise $x^*=x_2$.
Obviously, $S=\{x,y,w,x_4,x^*,y^*\}$ is full and we set $C_6=xx^*x_4y^*ywx$.
Otherwise, $x_4$ is  false and hence $x_3, y_3$ are true.
If $x_2$ is true, then  $C_6=xy_3ywx_2x_3x$ is a $6$-cycle;
hence we may assume $x_2$ is false.

\begin{claim}\label{T_{x_2}}
$x_2\in X_2^a$ and $\tau(f_w^{d(w)-3}\rightarrow T_{x_2})\ge \frac{1}{3}$.
\end{claim}
\proof
If $d(f_w^{d(w)-3})=3$, then $w_{d(w)-3}$ is true and we set $C_6=xy_3yx_3ww^{d(w)-3}x$.
If  $(f_x^2)_{x_2x_3}=[x_3x_2s]$ with $s$  true,
then $S=\{x,y,w,x_3,y_3,s\}$ is full and we set $C_6=xy_3ywx_3sx$.
Thus $d(f_w^{d(w)-3}), d((f_x^2)_{x_2x_3})\ge4$, and hence $x_2\in X_2^a$.

If $m_{4^+}(w)\ge4$, then $\tau(w\rightarrow T_v)\ge\frac{d(w)-4}{d(w)-4}=1$ by (R1),
and consequently $w'(T_v)\ge-2+\frac{2}{3}+\frac{1}{2}+1=\frac{1}{6}$.
Hence $m_{4^+}(w)=3$.
Write $f_w^{d(w)-3}=[ww_{d(w)-3}s_1\cdots s_{t-3}x_2]$ with $d(f_w^{d(w)-3})=t$.
If $t=4$, then $f_w^{d(w)-3}=[ww_{d(w)-3}s_1x_2]$.
Assume  $w_{d(w)-3}$  is true, then $S=\{x,y,w,x_3,s_1,w_{d(w)-3}\}$ is full and  we set $C_6=xx_3yww_{d(w)-3}s_1x$;
otherwise $w_{d(w)-3}$ is false,  $w_{d(w)-4}$ is true, and we set $C_6=xx_3yww_{d(w)-4}s_1x$ similarly.
Hence $t\ge5$.

If $t=5$, then $f_w^{d(w)-3}=[ww_{d(w)-3}s_1s_2x_2]$.
Assume  $s_1$ and $w_{d(w)-3}$ are true,
then Claim \ref{3path} applied to $P=s_1w_{d(w)-3}wxy$ gives $s_1\ne y$; thus $S=\{x,y,w,s_1,s_2,w_{d(w)-3}\}$ is full and $C_6=xyww_{d(w)-3}s_1s_2x$.
Assume  $s_1$ is true and $w_{d(w)-3}$ is false, then $w_{d(w)-4}$ is true;
we set $C_6=xyww_{d(w)-4}s_1s_2x$ if  $s_1\ne y$, and $C_6=yw_{d(w)-4}wx_3xy_3y$ if  $s_1= y$.
Assume $s_1$ is false and $w_{d(w)-3}$ is true, let $g=(f_w^{d(w)-3})_{s_1w_{d(w)-3}}$.
If $d(g)=3$, write $g=[s_1w_{d(w)-3}s']$, then $s'\ne y$ and we set $C_6=xyww_{d(w)-3}s's_2x$.
Hence $d(g)\ge4$, so  $m_{4^+}(f_w^{d(w)-3})\ge2$.
By (R2), $\tau(f_w^{d(w)-3}\rightarrow T_{x_2})\ge \frac{5-4}{5-2}=\frac{1}{3}$.
If $t\ge6$, then $m_{4^+}(f_w^{d(w)-3})\ge1$ and  $\tau(f_w^{d(w)-3}\rightarrow T_{x_2})\ge \frac{d(f_w^{d(w)-3})-4}{d(f_w^{d(w)-3})-1}=\frac{2}{5}$ by (R2).
Consequently, $\tau(f_w^{d(w)-3}\rightarrow T_{x_2})\ge \frac{1}{3}$.\qed

\begin{claim}\label{f_y^4}
$y_4\in X_1$ and $\tau(f_y^4\rightarrow T_{y_4})\ge\frac{2}{3}$.
\end{claim}

\proof By Claims \ref{6+v}  and \ref{T_{x_2}}, $\tau(w\rightarrow T_v)\ge\frac{2}{3}$.
We claim  that $d(y)=6$ and $y_4$ is false.
Otherwise, define $y^*=y_4$ if $y_4$ is true, and $y^*=y_5$ if $y_4$ is false.
Then $S=\{x,y,w,x_3,y_3,y^*\}$ is  full, and we  set $C_6=xy_3y^*ywx_3x$.
Let $y,y_3,s_1,s_2$ be the neighbors of $y_4$ and $z,y,u_1,u_2$ be the neighbors of $y_5$ in $H$, respectively.
By (P1), $s_i$ and $u_i$ are true for $i=1,2$.
Suppose that $(f_y^3)_{y_3y_4}=[y_4y_3s_1]$.
Then  $S=\{x,y,w,x_3,y_3,s_1\}$ is full, and  we set $C_6=xy_3s_1ywx_3x$.
Hence  $d((f_y^3)_{y_3y_4})\ge4$.

If  $s_2=u_1$, i.e., $f_y^4=[yy_4s_2y_5]$.
It is easy to check that $S=U\cup\{y_3,s_2\}$ is full  when $s_2\ne w$,
hence we  set $C_6=xzs_2y_3ywx$.
Otherwise, $s_2=w$, then   $U\cup\{x_3,y_3\}$ is  full and we set $C_6=wx_3yy_3xzs_2$.
If $d(f_y^4)=5$, then $f_y^4=[yy_4s_2u_1y_5]$.
Obviously, $S=\{x,y,z,y_3,u_1,s_2\}$ is full and hence $C_6=xyy_3s_2u_1zx$.
Thus, $d(f_y^4)\ge6$.

Assume $g_1=(f_y^4)_{y_4s_2}$ is a $3$-face. If $s_2\ne w$, then $S=\{x,y,w,y_3,s_1,s_2\}$ is full and set $C_6=xy_3s_2s_1ywx$; if $s_2=w$, then $S=\{x,y,w,x_3,y_3,s_1\}$ is full and set $C_6=ws_1yy_3xx_3w$. Both give a $6$-cycle, so $d(g_1)\ge4$.
Assume $g_2=(f_y^4)_{y_5u_1}$ is a $3$-face. If $u_1\ne y_3$, then $S=\{x,y,z,y_3,u_1,u_2\}$ is full and set $C_6=xy_3yu_2u_1zx$; if $u_1=y_3$, then $S=\{x,y,w,x_3,y_3,u_2\}$ is full and set $C_6=y_3u_2ywx_3xy_3$. Hence $d(g_2)\ge4$.
Consequently $y_4\in X_1$ and $m_{4^+}(f_y^4)\ge3$.
By (R2),  $\tau(f_y^4\rightarrow T_{y_4})\ge \frac{d(f_y^4)-4}{d(f_y^4)-3}\ge\frac{2}{3}$.
\qed

\medskip

Let $h=(f_y^2)_{x_4y_3}$.
For each $u\in\{x_3,y_3\}$, if $u\in V^o$ or $u\in V^i$ with $m_{4^+}(u)\ge2$, then by Claims \ref{T_{x_2}}--\ref{f_y^4} and (R0)--(R1), we have $\sigma(T_{x_2})\ge-2+\frac{2}{3}\times2+\frac{1}{2}+\frac{1}{3}=\frac{1}{6}$ and $\sigma(T_{y_4})\ge-1+\frac{1}{2}\times2+\frac{2}{3}=\frac{2}{3}$.
After (R0)--(R2) are implemented,  if $T_{x_4}$ is unsaturated, then by (r2)  $\tau(T_{y_4}\rightarrow T_{x_4})\ge\frac{2}{3}$.
Hence,  by (r2), $\tau(T_{x_4}\rightarrow T_v)\ge-3+1+\frac{2}{3}\times2+\frac{1}{2}\times2=\frac{1}{3}$ when $d(h)\ge4$, and $\tau(T_{x_4}\rightarrow T_v)\ge\frac{1}{2}(-4+1+\frac{2}{3}\times2+1\times2)=\frac{1}{6}$ when $d(h)=3$.
 If $T_{x_4}$ is saturated, then $\tau(T_{x_2}\rightarrow T_v)\ge\frac{1}{6}$ by (r2). In both subcases, $w'(T_v)\ge-2+\frac{2}{3}\times2+\frac{1}{2}+\frac{1}{6}=0$.
 Hence $d(h)=d(h_{y_3x_3})=3$; write $h_{y_3x_3}=[y_3x_3u']$.
If $u'$ is true, then $S=\{x,y,w,x_3,y_3,u'\}$ is full and  we set $C_6=xwyy_3u'x_3x$.
If $u'$ is false, let $y_3,x_3,t_1,t_2$ be the neighbors of  $u'$ in $H$, and  let  $h_{y_3x_3},h_1,h',h_2$ be its incident faces.
Then one of $h_1$ and $h_2$ must be a $3$-face.
Let $i\in \{1,2\}$ and assume $d(h_i)=3$.
Then $S=\{x,y,w,x_3,y_3,t_i\}$ is full and we set $C_6=xyy_3t_ix_3wx$, see Fig.\,\ref{fig5}.

    \begin{figure}[H]
 	\centering
 \epsfig{file=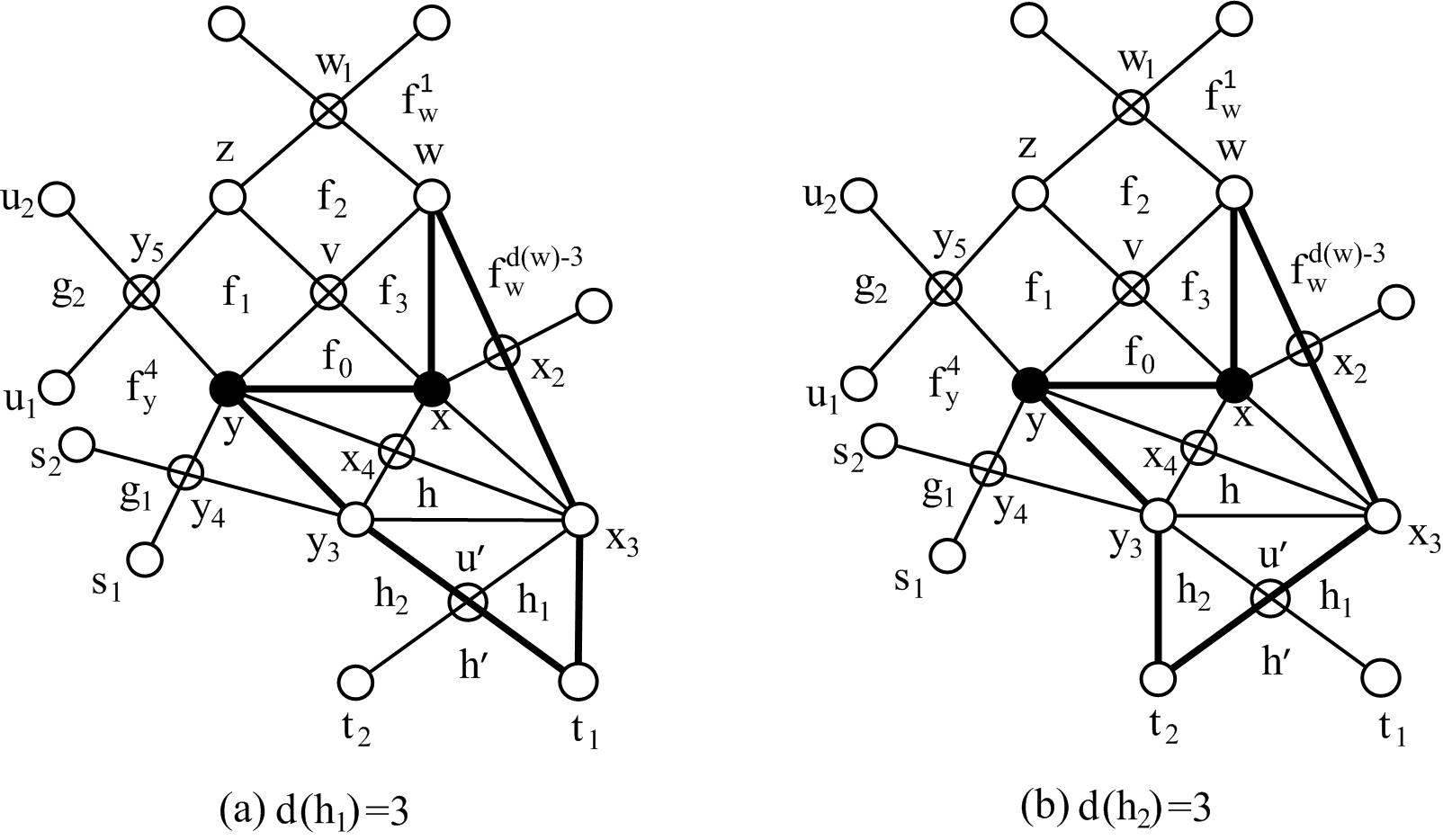, width=12.5cm}
 	\caption{ Two configurations in the proof of Case 2.2.2.}
 	\label{fig5}
 \end{figure}


\section{Proof of Theorem \ref{X_3}}
We assume, w.l.o.g., that $d(f_0)=d(f_1)=d(f_2)=3$ and $d(f_3)\ge4$.
Note that $w(T_v)=w(f_0)+w(f_1)+w(f_2)=(3-4)\times3=-3$.
In this case, $x=y_1$, $y=z_1$ and $z=w_1$.
By Claim \ref{6+v} and (R0), $\tau(u\rightarrow T_v)\ge \frac{2}{5}$ for $u\in\{x,w\}$ and $\tau(u\rightarrow T_v)\ge \frac{1}{3}\times2=\frac{2}{3}$ for $u\in\{y,z\}$.
If $w,x\in V^o$, then $\tau(u\rightarrow T_v)\ge \frac{1}{2}+\frac{7}{15}=\frac{29}{30}$ for  $u\in\{x,w\}$, and therefore $w'(T_v)\ge-3+\frac{29}{30}\times2+\frac{2}{3}\times2=\frac{4}{15}$ by (R0).
Hence, w.l.o.g., we may assume that $w\in V^i$.

Suppose that $x\in V^o$, then $\tau(x\rightarrow T_v)\ge \frac{1}{2}+\frac{7}{15}=\frac{29}{30}$ by (R0).
If $z\in V^o$, then $\tau(z\rightarrow T_v)\ge \frac{1}{2}+\frac{14}{15}=\frac{43}{30}$ by (R0) and Claim \ref{6+v}.
Therefore, $w'(T_v)\ge-3+\frac{29}{30}+\frac{2}{3}+\frac{43}{30}+\frac{2}{5}=\frac{7}{15}$.
Thus $z\in V^i$.
If $y\in V^o$, then $\tau(y\rightarrow T_v)\ge \frac{1}{2}+\frac{14}{15}=\frac{43}{30}$  by (R0).
Therefore, $w'(T_v)\ge-3+\frac{29}{30}+\frac{43}{30}+\frac{2}{3}+\frac{2}{5}=\frac{7}{15}$.
Hence $y\in V^i$.
By (R0), $\tau(x\rightarrow T_v)\ge \frac{14}{15}+\frac{7}{15}=\frac{21}{15}$.
Consequently, $w'(T_v)\ge-3+\frac{21}{15}+\frac{2}{3}\times2+\frac{2}{5}=\frac{2}{15}$.
Suppose that $x\in V^i$, then $\tau(x\rightarrow T_v)\ge \frac{2}{5}$ by Claim \ref{6+v}.
If $y,z\in V^o$, then $\tau(u\rightarrow T_v)\ge\frac{1}{2}+\frac{14}{15}=\frac{43}{30}$ for $u\in\{y,z\}$  by (R0).
Hence, $w'(T_v)\ge-3+\frac{2}{5}\times2+\frac{43}{30}\times2=\frac{2}{3}$.
If $y\in V^o$ and $z\in V^i$, then $\tau(y\rightarrow T_v)\ge\frac{14}{15}\times2=\frac{28}{15}$ by (R0) and $\tau(z\rightarrow T_v)\ge\frac{1}{3}\times2=\frac{2}{3}$ by Claim  \ref{6+v}. It follows that $w'(T_v)\ge-3+\frac{2}{5}\times2+\frac{28}{15}+\frac{2}{3}=\frac{1}{3}$.
By symmetry,  we may assume $x,y,z,w\in V^i$.

\medskip
\noindent{\bf Case 1.} $d(f_3)\ge5$.
\medskip

Suppose that $d(f_3)=5$ and write $f_3=[xvwsx_1]$.
If $s$ and $x_1$ are  true,  then $U\cup\{s,x_1\}$ is  full and we set $C_6=xyzwsx_1x$.
Otherwise, w.l.o.g., we assume that $s$ is true and $x_1$ is false by (P1).
Let $x,s,u_1,u_2$ be the neighbors of $x_1$ in $H$.
If $su_1\in E(G'')$,  then $U\cup\{s,u_1\}$ is  full and we set $C_6=xyzwsu_1x$.
If $xu_2\in E(G'')$,  then we set $C_6=xyzwsu_2x$ similarly.
Thus, $(f_3)_{sx_1}$ and $(f_3)_{xx_1}$ are $4^+$-faces.
By (R2), $\tau(f_3\rightarrow T_v)\ge \frac{5-4}{5-2}\times2=\frac{2}{3}$.
Suppose that $d(f_3)\ge6$, then  $\tau(f_3\rightarrow T_v)\ge \frac{d(f_3)-4}{d(f_3)}\times2=\frac{2}{3}$ by (R2).
Consequently, $\tau(f_3\rightarrow T_v)\ge \frac{2}{3}$.
If $d(u)\ge7$, or $d(u)=6$ with $m_{4^+}(u)\ge1$ for all $u\in\{y,z\}$, then $\tau(u\rightarrow T_v)\ge\frac{2}{5}\times2=\frac{4}{5}$ by Claim \ref{6+v}, and therefore $w'(T_v)\ge -3+\frac{4}{5}\times2+\frac{2}{5}\times2+\frac{2}{3}=\frac{1}{15}$.
Hence, by symmetry, we may assume $d(y)=6$ and $m_{4^+}(y)=0$.

Suppose  $y_2$ is true. Define $y^*=y_3$ if $y_3$   is true, otherwise $y^*=y_4$. Obviously, $U\cup\{y_2,y^*\}$ is full and  $C_6=xzwyy^*y_2x$ is a 6-cycle.
Thus $y_2$ is false, and $y_3$ is true  by (P1).
If $y_4$ is  true,  then  $C_6=xzwyy_4y_3x$ is a 6-cycle. Otherwise, $y_4$ is  false.
For each $u\in\{x,w\}$, if $\tau(u\rightarrow T_v)\ge\frac{1}{2}$, then $w'(T_v)\ge -3+\frac{1}{2}\times2+\frac{2}{3}\times3=0$.
Therefore, by Claim \ref{6+v}, at least one of $x$ and $w$ is a $6$-vertex with exactly one $4^+$-face.
If $d(x)=6$ and $m_{4^+}(x)=1$, then $U\cup\{y_3,x_3\}$ is full  and   we set  $C_6=xx_3ywzy_3x$.
If $d(w)=6$ and $m_{4^+}(w)=1$, then define $w^*=w_2$ if $w_2$   is true, otherwise $w^*=w_3$.
Thus $U\cup\{y_3,w^*\}$ is full and we  set $C_6=xzw^*wyy_3x$.

\medskip
\noindent{\bf Case 2.} $d(f_3)=4$ with $x_1$ is  true.
\medskip

It is easy to check  that $S=U\cup\{x_1\}$ is full.
If $f_w^1$ is a true $3$-face, then $S\cup\{w_2\}$ is full  and we may take $C_6=xyzw_2wx_1x$.
If $f_w^1$ is a false $3$-face, then $w_2$ is false and let $w,z,s_1,s_2$ be the neighbors of $w_2$ in $H$.
When $s_1z\in E(G'')$,  $S\cup\{s_1\}$ is  full and we set $C_6=xyzs_1wx_1x$.
Similarly, we set $C_6=xyzs_2wx_1x$ when $s_2w\in E(G'')$.
Hence for $u\in\{y,w\}$, by symmetry, $f_u^1$ is  either a $4^+$-face, or a false $3$-face with $u_2\in X_1\cup X_2^n$.

We claim that $f_z^1$ is also either a $4^+$-face, or a false $3$-face with $z_2\in X_1\cup X_2^n$.
If $f_z^1$ is a true $3$-face,
Claim \ref{3path} implies  that $z_2\ne x_1$ by taking $P=z_2yxx_1$ and hence $S\cup\{z_2\}$ is full. Then we set $C_6=xyz_2zwx_1x$.
If  $f_z^1$ is a false $3$-face, then  $z_2$ is a false vertex.
Let $z,y,u_1,u_2$ be the neighbors of $z_2$ in $H$.
If $u_2z\in E(G'')$, then Claim \ref{3path} applied to $P=u_2zwx_1$ yields $u_2\ne x_1$, so $S\cup\{u_2\}$ is full and we set $C_6=xyu_2zwx_1x$.
By symmetry, the same argument works if $u_1y\in E(G'')$.

Consequently, $m_{4^+}(u)\ge2$ for all $u\in U$.
By Claim \ref{6+v}, $\tau(u\rightarrow T_v)\ge\frac{1}{2}$ for $u\in\{x,w\}$, and $\tau(u\rightarrow T_v)\ge\frac{1}{2}\times2=1$ for $u\in\{y,z\}$.
Hence $w'(T_v)\ge-3+\frac{1}{2}\times2+1\times2=0$.

\medskip
\noindent{\bf Case 3.} $d(f_3)=4$ with $x_1$  is false.
\medskip

If at least one of $(f_3)_{xx_1}$ and $(f_3)_{wx_1}$ is a 3-face, then the discussion is similar to Case 2.
Otherwise, $(f_3)_{xx_1}$ and $(f_3)_{wx_1}$ are $4^+$-faces.
By Claim \ref{6+v}, $\tau(u\rightarrow T_v)\ge\frac{1}{2}$ for  each $u\in\{x,w\}$.
By symmetry, the proof is split into three subcases.

\medskip
\noindent{\bf Case 3.1.} $f_w^1$ is a true $3$-face.
\medskip

In this case, $f_w^1=[wzw_2]$ with $w_2$  true.
We claim that either $d(f_y^1)\ge4$,  or $d(f_y^1)=3$ with $y_2\in X_1$.
Otherwise, $d(f_y^1)=3$ and  $y_2\in V(G'')\cup X_2\cup X_3\cup X_4$.
If $y_2$ is true, then Claim \ref{3path} applied to $P=y_2yzw_2$ gives $y_2\ne w_2$; thus $U\cup\{w_2,y_2\}$ is full and we set $C_6=xzw_2wyy_2x$.
If $y_2$ is false, let $x,s_1,s_2,y$ be the neighbors of $y_2$ in $H$.
Assume $xs_1\in E(G'')$. Then Claim \ref{3path} applied to $P=s_1xyzw_2$ yields $s_1\ne w_2$, so $U\cup\{w_2,s_1\}$ is full and we set $C_6=xzw_2wys_1x$.
Assume $ys_2\in E(G'')$. Then similarly $U\cup\{w_2,s_2\}$ is full and we set $C_6=xzw_2wys_2x$.
Assume $s_1s_2\in E(G'')$. If $s_2\ne w$, then $U\cup\{s_1,s_2\}$ is full and we set $C_6=xzwys_1s_2x$; if $s_2=w$, then $U\cup\{w_2,s_1\}$ is full and we set $C_6=ww_2zxys_1w$.
Then we consider two possibilities as follows.

\medskip
\noindent{\bf Case 3.1.1.} $d(f_y^1)=3$ .
\medskip

Note that $y_2\in X_1$, hence $m_{4^+}(x)\ge3$ and $m_{4^+}(y)\ge1$.
By Claim  \ref{6+v}, $\tau(x\rightarrow T_v)\ge\frac{2}{3}$  and  $\tau(y\rightarrow T_v)\ge\frac{2}{5}\times2=\frac{4}{5}$.
We claim that $d(y)=6$ and $m_{4^+}(y)=1$. Otherwise, $\tau(y\rightarrow T_v)\ge\frac{1}{2}\times2=1$ and $\tau(y\rightarrow T_{y_2})\ge\frac{1}{2}$ by Claim \ref{6+v}.
After  (R0)--(R2)  are carried out, $\sigma(T_{y_2})\ge-1+\frac{1}{2}+\frac{2}{3}=\frac{1}{6}$ and hence $T_{y_2}$ is saturated.
By (r2), $\tau(T_{y_2}\rightarrow T_v)\ge\frac{1}{6}$.
Therefore, $w'(T_v)\ge-3+\frac{2}{3}\times2+1+\frac{1}{2}+\frac{1}{6}=0$.
If  $y_3$ and $y_4$ are  true,
then Claim \ref{3path} applied to $P=y_3yzw_2$ yields $y_3\ne w_2$. Hence $S=\{y,z,w,y_3,y_4,w_2\}$ is full and we set $C_6=yy_3y_4zw_2wy$.
By (P1), we consider the following two situations.

$\bullet$ Suppose that $y_3$ is  false  and $y_4$ is   true.
We claim that $d(f_y^2)\ge6$; otherwise  $4\le d(f_y^2)\le5$.
If $d(f_y^2)=4$, then  $f_y^2=[yy_2s_2y_3]$ with $s_2$  true.
If $s_2\ne w$, then $U\cup\{s_2,y_4\}$ is full and we set $C_6=xs_2y_4zwyx$; if $s_2=w$, then $U\cup\{y_4,w_2\}$ is full and we set $C_6=ww_2zxyy_4w$.
If $d(f_y^2)=5$, write $f_y^2=[yy_2s_2s_3y_3]$ with $s_2, s_3$ true.
Obviously, $S=\{x,y,z,s_2,s_3,y_4\}$ is full and we set $C_6=xs_2s_3y_4yzx$.
Hence $m_{4^+}(f_y^2)\ge1$,   and $\tau(f_y^2\rightarrow T_{y_2})\ge\frac{d(f_y^2)-4}{d(f_y^2)-1}\ge\frac{2}{5}$ by (R2).
After  (R0)--(R2) are   implemented, $\sigma(T_{y_2})\ge-1+\frac{2}{5}\times2+\frac{2}{3}=\frac{7}{15}$.
By (r1), $\tau(T_{y_2}\rightarrow T_v)\ge \frac{7}{15}$.
Consequently $w'(T_v)\ge-3+\frac{2}{3}\times2+\frac{4}{5}+\frac{1}{2}+\frac{7}{15}=\frac{1}{10}$.

$\bullet$ Suppose that $y_3$ is a true  and $y_4$ is false.
We  claim that $d(f_y^2)\ge5$; otherwise $d(f_y^2)=4$, and then $f_y^2=[yy_2s_2y_3]$.
Claim \ref{3path} applied to $P=s_2y_3yzw$ gives $s_2\ne w$, so $U\cup\{s_2,y_3\}$ is full and we set $C_6=xs_2y_3zwyx$.
Hence, $\tau(f_y^2\rightarrow T_{y_2})\ge\frac{d(f_y^2)-4}{d(f_y^2)-1}\ge\frac{1}{4}$ by (R2).
By (r2), $\tau(T_{y_2}\rightarrow T_v)\ge-1+\frac{2}{5}+\frac{2}{3}+\frac{1}{4}=\frac{19}{60}$.
Next, we claim that $d(z)=6$ and $m_{4^+}(z)=0$;
otherwise  $\tau(z\rightarrow T_v)\ge \frac{2}{5}\times2=\frac{4}{5}$ by Claim \ref{6+v}.
Hence, $w'(T_v)\ge-3+\frac{2}{3}+\frac{4}{5}\times2+\frac{1}{2}+\frac{19}{60}=\frac{1}{12}$.
Finally, $U\cup\{z_3,w_2\}$ is full and we set $C_6=xyz_3w_2wzx$.

\medskip
\noindent{\bf Case 3.1.2.} $d(f_y^1)\ge4$.
\medskip

Firstly, we claim that either $d(y)=6$ and $m_{4^+}(y)=1$, or $d(z)=6$ and $m_{4^+}(z)=0$.
Otherwise, by Claim \ref{6+v} we have $\tau(z\to f_w^1)\ge\frac25$, $\tau(z\to T_v)\ge\frac25\times2=\frac45$ and $\tau(y\to T_v)\ge\frac12\times2=1$.
Then, by (R0) and Claim \ref{6+v}, $\sigma(f_w^1)\ge-1+\frac25+\frac12+\frac13=\frac7{30}$, and by (r1) $\tau(f_w^1\to T_v)\ge\frac7{30}\times\frac13=\frac7{90}$.
Consequently, $w'(T_v)\ge-3+\frac23+1+\frac45+\frac12+\frac7{90}=\frac2{45}$, which completes the proof.
Thus we may assume that one of the two conditions holds.

$\bullet$ Suppose that $d(z)=6$ and $m_{4^+}(z)=0$.
Define $z^*=z_3$ if $z_3$ is true, otherwise $z^*=z_2$.
Then $U\cup\{w_2,z^*\}$ is full, and we obtain the $6$-cycle $C_6=xyww_2z^*zx$.

$\bullet$ Suppose that $d(y)=6$ and $m_{4^+}(y)=1$.
In this case, we have $d(z)\ge7$, or $d(z)=6$ with $m_{4^+}(z)\ge1$; otherwise we are back to the previous situation.
Hence, by the same reasoning as before we  have $\tau(z\to T_v)\ge\frac45$  and $\tau(f_w^1\to T_v)\ge\frac7{90}$.
Assume $y_4$ is true. Define $y^*=y_3$ if $y_3$ is true, otherwise $y^*=y_2$.
Applying Claim \ref{3path} to $P=w_2zyy^*$ yields $w_2\ne y^*$; thus $S=\{y,z,w,w_2,y_4,y^*\}$ is full and we set $C_6=yy^*y_4zw_2wy$.
Otherwise, $y_4$ is false and  then $y_3$ is true.
If $y_2$ is true, then $P=w_2zyy_2$ gives $w_2\ne y_2$, and we set $C_6=yy_2y_3zw_2wy$.
If $y_2$ is false, we claim that $d(f_y^1)\ge5$; otherwise $d(f_y^1)=4$ and $f_y^1=[xx_{d(x)-2}y_2y]$.
Taking $P=wzyxx_{d(x)-2}$, we obtain that $U\cup\{y_3,x_{d(x)-2}\}$ is full and we set $C_6=xywzy_3x_{d(x)-2}x$.
Hence $d(f_y^1)\ge5$, and by (R2) $\tau(f_y^1\to T_v)\ge\frac15$.
Consequently, $w'(T_v)\ge-3+\frac23+\frac45\times2+\frac12+\frac7{90}+\frac15=\frac2{45}$.


\medskip
\noindent{\bf Case 3.2.} $f_w^1$ is a false  $3$-face.
\medskip

By symmetry, $f_y^1$ is not a true 3-face.
Let $w,z,s_1,s_2$ be the neighbors of $w_2$ in $H$, and let $f_w^1,g_1,g_2,g_3$ be its incident faces.
We discuss the following subcases according to the degrees of $g_1$ and $g_3$.

\medskip
\noindent{\bf Case 3.2.1.} $d(g_1)=3$.
\medskip

If $d(f_y^1)=3$, i.e., $f_y^1=[xy_2y]$ with $y_2$ a false vertex.
Let $y,x,u_1,u_2$ be the neighbors of $y_2$ in $H$.
The analysis of $y_2$ is identical to that in Case 3.1, thus $y_2\in X_1$. Consequently $m_{4^+}(x)\ge3$ and $m_{4^+}(y)\ge1$.
By Claim \ref{6+v}, $\tau(x\to T_v)\ge\frac23$ and $\tau(y\to T_v)\ge\frac25\times2=\frac45$.

Suppose $d(f_z^1)\ge4$. Then by Claim \ref{6+v}, $\tau(y\to T_v)\ge\frac12\times2=1$ and $\tau(z\to T_v)\ge\frac25\times2=\frac45$.
We claim that $d(z)=6$ and $m_{4^+}(z)=1$; otherwise $\tau(z\to T_v)\ge\frac12\times2=1$ by Claim \ref{6+v}, and then $w'(T_v)\ge-3+\frac23+1\times2+\frac12=\frac16$.
If $z_2$ is true, then $U\cup\{s_1,z_2\}$ is full and we set $C_6=xyws_1z_2zx$.
If $z_2$ is false and $f_z^1=[yy_{d(y)-2}z_2z]$, then $U\cup\{s_1,y_{d(y)-2}\}$ is full and $C_6=xyy_{d(y)-2}s_1wzx$.
Hence $d(f_z^1)\ge5$, so by (R2) $\tau(f_z^1\to T_v)\ge\frac15$, and consequently $w'(T_v)\ge-3+\frac23+1+\frac45+\frac12+\frac15=\frac16$.
Thus we may assume $d(f_z^1)=3$, i.e., $f_z^1=[zyz_2]$.
Observing $z_2$, we have to handle two subcases.

{\bf (a)}   $z_2$ is true.

If $d(z)=6$ and $m_{4^+}(z)=0$, then $z_2s_1\in E(G'')$,  so $U\cup\{s_1,z_2\}$ is full and we set $C_6=xzws_1z_2yx$.
Hence,  by Claim \ref{6+v}, $\tau(z\rightarrow T_v)\ge\frac{2}{5}\times2=\frac{4}{5}$  and $\tau(z\rightarrow f_z^1)\ge\frac{2}{5}$.
We claim that $d(y)=6$ and $m_{4^+}(y)=1$;
otherwise $\tau(y\rightarrow T_v)\ge\frac{1}{2}\times2=1$  and $\tau(y\rightarrow f_z^1)\ge\frac{1}{2}$ by Claim \ref{6+v}.
After (R0)--(R2) are implemented, $\sigma(f_z^1)\ge-1+\frac{1}{2}+\frac{2}{5}+\frac{1}{3}=\frac{7}{30}$ by (R0) and Claim \ref{6+v}.
By (r1), $\tau(f_z^1\rightarrow T_v)\ge\frac{7}{30}\times\frac{1}{3}=\frac{7}{90}$.
Consequently, $w'(T_v)\ge-3+\frac{2}{3}+1+\frac{4}{5}+\frac{1}{2}+\frac{7}{90}=\frac{2}{45}$.

If $y_3$ is true, then applying Claim \ref{3path} to $P=y_3yzs_1$ gives $y_3\ne s_1$; thus $S=\{y,z,w,z_2,y_3,s_1\}$ is full and we set $C_6=yy_3z_2zs_1wy$.
If $y_3$ is false, first assume $d(f_y^1)=3$, so $f_y^1=[xy_2y]$ with $y_2\in X_1$.
As in Case 3.1.1, if $d(f_y^2)=4$ we obtain a $6$-cycle; hence we may assume $d(f_y^2)\ge5$, so by (R2) $\tau(f_y^2\to T_{y_2})\ge\frac{d(f_y^2)-4}{d(f_y^2)-1}\ge\frac14$, and then $\tau(T_{y_2}\to T_v)\ge\sigma(T_{y_2})\ge-1+\frac23+\frac25+\frac14=\frac{19}{60}$ by (r2).
Consequently $w'(T_v)\ge-3+\frac23+\frac45\times2+\frac12+\frac{19}{60}=\frac1{12}$.
If $d(f_y^1)\ge4$, then $y_2$ is true; applying Claim \ref{3path} to $P=y_2yzs_1$ yields $y_2\ne s_1$, so $S=\{y,z,w,s_1,z_2,y_2\}$ is full and  set $C_6=yy_2z_2zs_1wy$.

{\bf (b)}  $z_2$ is  false.

Let $z,y,t_1,t_2$ be the neighbors of $z_2$ in $H$.
Note that if $d(g_3)\ge4$ then $w'(T_v)\ge-3+\frac23\times3+\frac45=-\frac15$, and if $d(z)\ge7$ then $w'(T_v)\ge-3+\frac23+\frac45+\frac67+\frac12=-\frac{37}{210}>-\frac15$.
In either case, we claim that $d(y)=6$ with $m_{4^+}(y)=1$; otherwise $\tau(y\to T_v)\ge\frac12\times2=1$ by Claim \ref{6+v}, and then $w'(T_v)\ge-3+\frac23\times3+1=0$ when $d(g_3)\ge4$, and $w'(T_v)\ge-3+\frac23+1+\frac67+\frac12=\frac1{42}$ when $d(z)\ge7$.
Suppose  $d(f_y^1)=3$.
If $d(f_y^2)=4$, then Claim \ref{3path} applied to $P=u_2t_1yzw$ gives $u_2\ne w$, so $U\cup\{u_2,t_1\}$ is full and we set $C_6=xywzt_1u_2x$.
Thus $d(f_y^2)\ge5$, so $\tau(f_y^2\to T_{y_2})\ge\frac{d(f_y^2)-4}{d(f_y^2)-1}\ge\frac14$ by (R2).
After (R0)--(R2) are implemented, $\tau(T_{y_2}\to T_v)\ge\sigma(T_{y_2})\ge-1+\frac25+\frac23+\frac14=\frac{19}{60}>\frac15$ by (r2), hence $w'(T_v)\ge0$.
Suppose $d(f_y^1)\ge4$.
If $y_2$ is true, then similarly we set $C_6=yy_2t_1zs_1wy$.
If $y_2$ is false, we claim $d(f_y^1)\ge5$; otherwise $d(f_y^1)=4$ and $f_y^1=[xx_{d(x)-2}y_2y]$. As in Case 3.1.2, we set $C_6=xywzt_1x_{d(x)-2}x$ and $\tau(f_y^1\to T_v)\ge\frac15$ by (R2), which gives $w'(T_v)\ge0$.
Thus we must have $d(g_3)=3$ and $d(z)=6$.

If $d(f_z^2)=3$, then $t_2=s_1$ and we set $C_6=xys_1ws_2zx$.
If $d(f_z^2)=4$, then $f_z^2=[zz_2t_2s_1]$ and $U\cup\{s_1,t_2\}$ is full, so we set $C_6=xyt_2s_1wzx$.
Hence $d(f_z^2)\ge5$.
Suppose $t_1t_2\in E(G'')$. If $t_2\ne s_1$, set $C_6=yt_2t_1zs_1wy$; if $t_2=s_1$, set $C_6=s_1yxzs_2ws_1$. Thus $t_1t_2\notin E(G'')$.
By (R2), $\tau(f_z^2\to T_{z_2})\ge\frac{d(f_z^2)-4}{d(f_z^2)-1}\ge\frac14$.
If $(f_z^1)_{yz_2}$ is a $4^+$-face, then $\sigma(T_{z_2})\ge-1+\frac25+\frac12+\frac14=\frac3{20}$, so $\tau(T_{z_2}\to T_v)\ge\frac3{20}$ by (r2), and consequently
$w'(T_v)\ge-3+\frac23+1+\frac45+\frac12+\frac3{20}=\frac7{60}$.
Thus $(f_z^1)_{yz_2}$ is a $3$-face.

$\bullet$ Suppose $d(f_y^1)=3$. We claim $d(y)=6$ and $m_{4^+}(y)=1$; otherwise $\tau(y\to T_v)\ge1$ and $\tau(y\to T_{y_2})\ge\frac12$ by Claim \ref{6+v}. After (R0)--(R2) are carried  out, we get $\tau(T_{y_2}\to T_v)\ge\sigma(T_{y_2})\ge-1+\frac23+\frac12=\frac16$ by (r2), hence $w'(T_v)\ge -3+\frac23+1+\frac45+\frac12+\frac16=\frac2{15}$.
If $d(f_y^2)=4$, we obtain a $6$-cycle as before; hence $d(f_y^2)\ge5$ and $\tau(T_{y_2}\to T_v)\ge\frac{19}{60}$ by (r2), giving $w'(T_v)\ge-3+\frac23+\frac45\times2+\frac12+\frac{19}{60}=\frac1{12}$.

$\bullet$ Suppose $d(f_y^1)\ge4$. We discuss according to the degree of $y$.
If $d(y)\ge7$, we claim $m_{4^+}(y)=1$; otherwise $\tau(y\to T_v)\ge\frac35\times2=\frac65$ by Claim \ref{6+v}, and then $w'(T_v)\ge-3+\frac23+\frac65+\frac45+\frac12=\frac16$.
Define $y^*=y_{d(y)-4}$ if $y_{d(y)-4}$  is true, otherwise $y^*=y_{d(y)-5}$. Applying Claim \ref{3path} to $P=y^*yzs_1$ gives $y^*\ne s_1$; hence $S=\{y,z,w,t_1,s_1,y^*\}$ is full and we set $C_6=yy^*t_1zs_1wy$.
Thus $d(y)=6$. First suppose $d(f_y^2)=3$, i.e., $f_y^2=[t_1yy_2]$.
If $y_2$ is true, or $y_2$ is false and $d(f_y^1)=4$, we find a $6$-cycle as before.
Hence we may assume $y_2$ is false and $d(f_y^1)\ge5$.
Let $y,r_1,r_2,t_1$ be the neighbors of $y_2$ in $H$.
If $r_1r_2\in E(G'')$, then $S=\{x,y,z,t_1,r_1,r_2\}$ is full and we set $C_6=xyr_2r_1t_1zx$; thus $r_1r_2\notin E(G'')$.
By (R2), $\tau(f_y^1\to T_v)\ge\frac{d(f_y^1)-4}{d(f_y^1)-1}\ge\frac14$, and consequently $w'(T_v)\ge-3+\frac23+\frac45\times2+\frac12+\frac14=\frac1{60}$.
Otherwise $d(f_y^2)\ge4$.
By Claim \ref{6+v} and (R0), $\tau(y\to T)\ge1$ for $T\in\{T_v,T_{z_2}\}$, and $\tau(t_1\to T_{z_2})\ge\frac12$. After (R0)--(R2) are implemented, $\tau(T_{z_2}\to T_v)\ge\sigma(T_{z_2})\ge-2+1+\frac12+\frac25+\frac14=\frac3{20}$.
Hence $w'(T_v)\ge-3+\frac23+1+\frac45+\frac12+\frac3{20}=\frac7{60}$.


\medskip
\noindent{\bf Case 3.2.2.} $d(g_1)\ge4$ and $d(g_3)=3$.
\medskip

Assume $d(f_y^1)=3$, i.e., $f_y^1=[xy_2y]$ with $y_2$  a false vertex.
Let $y,x,u_1,u_2$ be the neighbors of $y_2$ in $H$.
By the preceding discussion, we may assume $u_1u_2\notin E(G'')$ and $d(f_y^2)\ge4$.
Hence, if $d(f_y^1)=3$, then $y_2\in X_1\cup X_2^a$.

First suppose $d(f_y^1)=3$ and $y_2\in X_2^a$.
If $u_1\ne s_2$, then $U\cup\{u_1,s_2\}$ is full and we set $C_6=xu_1yws_2zx$; hence $u_1=s_2$.
If $m_{4^+}(u)\ge2$ for all $u\in U$, then by Claim \ref{6+v} we have $w'(T_v)\ge-3+1\times2+\frac12\times2=0$.
Thus we consider the case $d(f_z^1)=3$, i.e., $f_z^1=[zyz_2]$.
If $z_2$ is true, then $U\cup\{s_2,z_2\}$ is full and we set $C_6=s_2xyz_2zws_2$.
If $z_2$ is false, then at least one of $(f_z^1)_{zz_2}$ and $(f_z^1)_{yz_2}$ is a $3$-face.
Let $z,y,t_1,t_2$ be the neighbors of $z_2$ in $H$.
If $yt_1\in E(G'')$, then $U\cup\{s_2,t_1\}$ is full and we set $C_6=s_2xyt_1zws_2$.
If $zt_2\in E(G'')$, then $U\cup\{s_2,t_2\}$ is full and we set $C_6=s_2xyt_2zws_2$.
Thus either $d(f_y^1)\ge4$, or $d(f_y^1)=3$ with $y_2\in X_1$.
Consequently $m_{4^+}(x)\ge3$ and $m_{4^+}(y)\ge1$.
By Claim \ref{6+v},
$\tau(x\rightarrow T_v)\ge\frac{2}{3}$ and $\tau(u\rightarrow T_v)\ge\frac{2}{5}\times2=\frac{4}{5}$ for $u\in\{y,z\}$.

Suppose $d(f_y^1)=3$ with $y_2\in X_1$.
If $d(f_y^2)\ge5$, then by (R2) $\tau(f_y^2\to T_{y_2})\ge\frac{d(f_y^2)-4}{d(f_y^2)-1}\ge\frac14$, so by (r1) we have
$\tau(T_{y_2}\to T_v)\ge\sigma(T_{y_2})\ge-1+\frac23+\frac25+\frac14=\frac{19}{60}$.
Consequently $w'(T_v)\ge-3+\frac23+\frac45\times2+\frac12+\frac{19}{60}=\frac1{12}$.
Hence $d(f_y^2)=4$, i.e., $f_y^2=[yy_2u_2y_3]$.
If $y_3$ is true, then Claim \ref{3path} applied to $P=u_2y_3yzw$ gives $u_2\ne w$, so $U\cup\{u_2,y_3\}$ is full and we set $C_6=xu_2y_3ywzx$.
Assume $y_3$ is false. We claim that $d(y)=6$ and $m_{4^+}(y)=1$; otherwise $\tau(y\to T_v)\ge\frac12\times2=1$ and $\tau(y\to T_{y_2})\ge\frac12$ by Claim \ref{6+v}.
After (R0)--(R2) are implemented, $\tau(T_{y_2}\to T_v)\ge\sigma(T_{y_2})\ge-1+\frac23+\frac12=\frac16$, hence $w'(T_v)\ge-3+\frac23+1+\frac45+\frac12+\frac16=\frac2{15}$.
Now $y_4$ is true. If $u_2\ne w$, then $U\cup\{u_2,y_4\}$ is full and we set $C_6=xywzy_4u_2x$; if $u_2=w$, then $U\cup\{s_2,y_4\}$ is full and $C_6=wxyy_4zs_2w$.

Finally  consider the case $d(f_y^1)\ge4$.
If $d(u)\ge7$, or $d(u)=6$ with $m_{4^+}(u)\ge2$ for  all $u\in\{y,z\}$, then $\tau(u\rightarrow T_v)\ge\frac{1}{2}\times2=1$ by Claim  \ref{6+v}.
Therefore, $w'(T_v)\ge-3+\frac{2}{3}+1\times2+\frac{1}{2}=\frac{1}{6}$.
Thus we discuss the following two possibilities.

$\bullet$ Suppose $d(y)=6$ and $m_{4^+}(y)=1$.
If $y_4$ is true, define $y^*=y_3$ if $y_3$ true, otherwise $y^*=y_2$.
Applying Claim \ref{3path} to $P=y^*yzws_2$ yields $y^*\ne s_2$, hence $S=\{y,z,w,y_4,s_2,y^*\}$ is full and we set $C_6=yy^*y_4zs_2wy$.
If $y_4$ is false, then $y_3$ is true.
If $y_2$ is true, then Claim \ref{3path} gives $C_6=yy_2y_3zs_2wy$, hence $y_2$ is false.
As in Case 3.2.1, we obtain $d(f_y^1)\ge5$ and $m_{4^+}(f_y^1)\ge1$, so $\tau(f_y^1\to T_v)\ge\frac14$  by (R2).
Thus $w'(T_v)\ge-3+\frac23+\frac45\times2+\frac12+\frac14=\frac1{60}$.

$\bullet$ Suppose $d(z)=6$ and $m_{4^+}(z)=1$.
Then by Claim \ref{6+v}, $\tau(y\to T_v)\ge\frac12\times2=1$.
If $z_2$ is true, then after (R0)--(R2) we have $\sigma(f_z^1)\ge-1+\frac12+\frac25+\frac13=\frac7{30}$, and $\tau(f_z^1\to T_v)\ge\frac7{30}\times\frac13=\frac7{90}$ by (r1); hence $w'(T_v)\ge-3+\frac23+1+\frac45+\frac12+\frac7{90}=\frac2{45}$.
Thus $z_2$ is false, so $z_3$ is true.
If $d(g_2)=3$, then $S=\{y,z,w,z_3,s_1,s_2\}$ is full and we set $C_6=yz_3zs_2s_1wy$.
Otherwise $d(g_2)\ge4$.
We claim that $d(g_1)\ge5$ and $\tau(g_1\to T_{w_2})\ge\frac13$; if $d(g_1)\ge6$ this is immediate by (R2).
So assume $4\le d(g_1)\le5$.
If $d(g_1)=4$, write $g_1=[zz_3s_1w_2]$.
Applying Claim \ref{3path} to $P=s_1z_3zyx$ yields $s_1\ne x$, so $U\cup\{s_1,z_3\}$ is full and we set $C_6=xyz_3s_1wzx$.
If $d(g_1)=5$, let $g_1=[zz_3qs_1w_2]$.
If $q$ is true, then $S=\{y,z,w,z_3,s_1,q\}$ is full and we set $C_6=zyz_3qs_1wz$.
If $q$ is false, let $s_1,z_3,q_1,q_2$ be the neighbors of $q$ in $H$.
If $(g_1)_{qz_3}$ is a $3$-face, then Claim \ref{3path} gives $C_6=zyz_3q_1s_1wz$; hence $(g_1)_{qz_3}$ is a $4^+$-face and $m_{4^+}(g_1)\ge2$.
By (R2), $\tau(g_1\to T_{w_2})\ge\frac{5-4}{5-2}=\frac13$.
After (R0)--(R2) are carried  out, $\sigma(T_{w_2})\ge-2+1+\frac25\times2+\frac13=\frac2{15}$, and  $\tau(T_{w_2}\to T_v)\ge\frac2{15}\cdot\frac12=\frac1{15}$ by (r2).
Consequently $w'(T_v)\ge-3+\frac23+1+\frac45+\frac12+\frac1{15}=\frac1{30}$.

\medskip
\noindent{\bf Case 3.2.3.} $d(g_1), d(g_3)\ge4$.
\medskip

In this case, either $d(f_y^1)\ge4$, or $d(f_y^1)=3$ with $y_2\in X_1\cup X_2^n$.
By Claim \ref{6+v}, $\tau(u\rightarrow T_v)\ge\frac{2}{3}$ for $u\in\{x,w\}$, and $\tau(u\rightarrow T_v)\ge\frac{2}{5}\times2=\frac{4}{5}$ for $u\in\{y,z\}$.
If $d(g_2)\ge4$, then $w_2\in X_1$ and after (R0)--(R2) we have $\sigma(T_{w_2})\ge-1+\frac{2}{3}+\frac{2}{5}=\frac{1}{15}$.
By (r2), $\tau(T_{w_2}\rightarrow T_v)\ge\frac{1}{15}$, hence $w'(T_v)\ge-3+\frac{2}{3}\times2+\frac{4}{5}\times2+\frac{1}{15}=0$.
Thus we may assume $d(g_2)=3$.
We claim that $d(z)=6$ and $m_{4^+}(z)=1$;
otherwise $\tau(z\rightarrow T_v)\ge\frac{1}{2}\times2=1$ by Claim \ref{6+v},
and then  $w'(T_v)\ge-3+\frac{2}{3}\times2+1+\frac{4}{5}=\frac{2}{15}$.
Define $z^*=z_2$ if $z_2$ is true, otherwise $z^*=z_3$.
Then $S=\{y,z,w,s_1,s_2,z^*\}$ is full and we set $C_6=yz^*zs_2s_1wy$.

\medskip
\noindent{\bf Case 3.3.} $f_w^1$ and $f_y^1$ are $4^+$-faces.
\medskip

By Claim \ref{6+v}, $\tau(u\rightarrow T_v)\ge\frac{2}{3}$ for $u\in\{x,w\}$, and $\tau(u\rightarrow T_v)\ge\frac{2}{5}\times2=\frac{4}{5}$ for $u\in\{y,z\}$.
As in Case 3.2.3, we obtain $d(y)=d(z)=6$ and $m_{4^+}(y)=m_{4^+}(z)=1$.
Suppose first that $z_2=y_4$ is true.
Define $y^*=y_3$ if $y_3$ is true, otherwise $y^*=y_2$; define $z^*=z_3$ if $z_3$ is true, otherwise $z^*=z_4$.
Then $S=\{x,y,z,z_2,z^*,y^*\}$ is full, and we set $C_6=xyy^*z_2z^*zx$.
Now suppose $z_2=y_4$ is false. Then $y_3$ and $z_3$ are true.
If $y_2$ and $z_4$ are true, then Claim \ref{3path} applied to $P=y_2yzz_4$ gives $y_2\ne z_4$; hence $S=\{y,z,y_2,y_3,z_3,z_4\}$ is full and we set $C_6=yy_2y_3zz_4z_3y$.
Thus, without loss of generality, we may assume $z_4$ is false.
If $d(f_w^1)=4$, i.e., $f_w^1=[wzz_4w_2]$, then $w_2$ is true by (P1). Applying Claim \ref{3path} to $P=w_2wzyy_3$ yields $w_2\ne y_3$; hence $S=\{y,z,w,y_3,z_3,w_2\}$ is full and we set $C_6=yy_3zz_3w_2wy$.
If $d(f_w^1)\ge5$, then $\tau(f_w^1\to T_v)\ge\frac15$ by (R2), and consequently
$w'(T_v)\ge-3+\frac23\times2+\frac45\times2+\frac15=\frac2{15}$.

\section{Proof of Theorem \ref{X_4}}

Note that $d(f_i)=3$ for all $i\in[0,3]$, hence $w(T_v)=w(f_0)+w(f_1)+w(f_2)+w(f_3)=(3-4)\times4=-4$.
In this case, $x=y_1$, $y=z_1$, $z=w_1$ and $w=x_1$.
By Claim \ref{6+v},  $\tau(u\rightarrow T_v)\ge \frac{1}{3}\times2=\frac{2}{3}$ for each $u\in U$.
If $U\subseteq V^o$, then $G''=K_4$, which contradicts the definition of $G''$.
Thus, without loss of generality, we may assume $x\in V^i$.
If $y,w\in V^o$, then $\tau(u\rightarrow T_v)\ge\frac{1}{2}+\frac{14}{15}=\frac{43}{30}$ by (R0) for $u\in\{y,w\}$, and it follows that $w'(T_v)\ge-4+\frac{2}{3}\times2+\frac{43}{30}\times2=\frac{1}{5}$.
Hence $y\in V^i$ by symmetry.
Similarly, if $z,w\in V^o$, then the same calculation gives $w'(T_v)\ge\frac15$; so $w\in V^i$ by symmetry.

Now suppose $z\in V^o$.
By (R0), $\tau(z\to T_v)\ge 2\times\frac{14}{15}=\frac{28}{15}$.
We claim that $d(u)=6$ and $m_{4^+}(u)=0$ for all $u\in\{x,y,w\}$; otherwise, by Claim \ref{6+v}, we  have $w'(T_v)\ge -4+\frac{28}{15}+\frac45+\frac23\times2=0$.
Assume $x_2$ is true. Define $x^*=x_3$ if $x_3$ is true, otherwise $x^*=x_4$.
Then $U\cup\{x_2,x^*\}$ is full, and we set $C_6=xyzwx_2x^*x$.
Thus $x_2$ is false; by symmetry, $x_4$ is also false.
By (P1), $x_3$ and $w_3$ are true, and $U\cup\{x_3,w_3\}$ is full, so we set $C_6=xx_3yzww_3x$.

It remains to consider the case $U\subseteq V^i$.
If $m_{4^+}(u)\ge2$ for all $u\in U$, then $\tau(u\to T_v)\ge \frac12\times2=1$ by Claim \ref{6+v}, so $w'(T_v)\ge -4+1\times4=0$.
By symmetry, we may assume that $d(f_w^1)=3$, i.e., $f_w^1=[wzw_2]$.
We then consider the following two subcases.

\medskip
\noindent{\bf Case 1.} $w_2$ is  true.
\medskip

First, we give a key claim.

\begin{claim}\label{X_4_1}
If $d(f_u^1)=3$ for $u\in\{x,y,z\}$, then $u_2\in X_1$.
\end{claim}

\proof Suppose $f_x^1=[wx_2x]$.
If $x_2$ is true, then $U\cup\{x_2,w_2\}$ is full  and we set $C_6=xyzw_2wx_2x$.
Otherwise, $x_2$ is false. Let $x,w,s_1,s_2$ be the neighbors of $x_2$ in $H$.
If $ws_1\in E(G'')$, then $C_6=xyzw_2ws_1x$;
if $xs_2\in E(G'')$, then $C_6=xyzw_2ws_2x$;
and if $s_1s_2\in E(G'')$, then $C_6=xyzws_2s_1x$.
By symmetry, $x_2, z_2\in X_1$.

Now suppose $f_y^1=[xy_2y]$.
If $y_2$ is true, then Claim \ref{3path} applied to $P=y_2xww_2$ gives $y_2\ne w_2$, so $U\cup\{w_2,y_2\}$ is full and we set $C_6=xy_2yzw_2wx$.
If $y_2$ is false, let $y,x,t_1,t_2$ be the neighbors of $y_2$ in $H$.
If $xt_1\in E(G'')$, then Claim \ref{3path} applied to $P=t_1xww_2$ yields $t_1\ne w_2$, so $U\cup\{w_2,t_1\}$ is full and we set $C_6=xt_1yzw_2wx$.
Similarly, if $yt_2\in E(G'')$ or $t_1t_2\in E(G'')$, we obtain a $6$-cycle.
Thus $y_2\in X_1$.\qed

\medskip

By Claim \ref{X_4_1}, we have $m_{4^+}(u)\ge2$ for $u\in\{x,y\}$ and $m_{4^+}(u)\ge1$ for $u\in\{z,w\}$.
By Claim \ref{6+v}, $\tau(u\to T_v)\ge\frac12\times2=1$ for $u\in\{x,y\}$, and $\tau(u\to T_v)\ge\frac25\times2=\frac45$ for $u\in\{z,w\}$.
If $d(u)\ge7$, or $d(u)=6$ with $m_{4^+}(u)\ge2$ for all $u\in\{z,w\}$, then $\tau(u\to T_v)\ge\frac12\times2=1$ by Claim \ref{6+v}, hence $w'(T_v)\ge-4+1\times4=0$.
Thus, without loss of generality, we may assume $d(w)=6$ and $m_{4^+}(w)=1$.
Suppose $d(f_x^1)\ge4$. Define $w^*=w_3$ if $w_3$ true, otherwise $w^*=w_4$. Then $U\cup\{w_2,w^*\}$ is full, and we set $C_6=xyzw_2w^*wx$.

Thus we may assume $d(f_x^1)=3$, so by Claim \ref{X_4_1} we have $x_2\in X_1$.
If $w_3$ is true, then $C_6=xyzw_2w_3wx$.
If $w_3$ is false, let $w,w_2,u_1,u_2$ be the neighbors of $w_3$ in $H$.
If $d(f_w^3)=4$, i.e., $u_2=s_1$, then $U\cup\{w_2,u_2\}$ is full and we set $C_6=xyzww_2u_2x$.
If $d(f_w^3)=5$, i.e., $u_2s_1\in E(G'')$, then $S=\{x,z,w,w_2,u_2,s_1\}$ is full and we set $C_6=xwzw_2u_2s_1x$.
Thus $d(f_w^3)\ge6$ and $m_{4^+}(f_w^3)\ge1$, so $\tau(f_w^3\to T_{x_2})\ge\frac{d(f_w^3)-4}{d(f_w^3)-1}\ge\frac25$ by (R2).
After (R0)--(R2) are implemented, $\sigma(T_{x_2})\ge-1+\frac25\times2+\frac12=\frac3{10}$, and $\tau(T_{x_2}\to T_v)\ge\frac3{10}$ by (r2).
We claim that $d(z)=6$ and $m_{4^+}(z)=1$; otherwise $\tau(z\to T_v)\ge\frac12\times2=1$ by Claim \ref{6+v}, and then $w'(T_v)\ge-4+1\times3+\frac45+\frac3{10}=\frac1{10}$.
By the same argument as above, we derive that $d(f_z^1)=3$ with $z_2\in X_1$ and $d(f_z^2)\ge6$ by symmetry.
Similarly, after (R0)--(R2) we have $\tau(T_{z_2}\to T_v)\ge\frac3{10}$.
Consequently, $w'(T_v)\ge-4+1\times2+\frac45\times2+\frac3{10}\times2=\frac15$.

\medskip
\noindent{\bf Case 2.} $w_2$ is false.
\medskip

By symmetry, we may assume that for each $u\in U\setminus\{w\}$, either $d(f_u^1)\ge4$, or $d(f_u^1)=3$ with $u_2$ a false vertex.
Analogous to the proof of Claim \ref{X_4_1}, we have the following claim.

\begin{claim}\label{X_4_2}
 Assume that $d(f_w^2)=d(f_u^1)=3$.

 {\em (1)} If $u\in\{y,z\}$, then $u_2\in X_1$.

 {\em (2)} If $u=x$ and $d(w)\ge7$, then $x_2\in X_1$.

 {\em (3)} If $u=x$ and $d(w)=6$, then $x_2\in X_1\cup  X_2^a$; moreover, in the case $x_2\in X_2^a$, we have $d(f_w^3)=3$.
\end{claim}

It follows that $w_2\in X_1\cup X_2^a\cup X_3$; otherwise a $6$-cycle would be found.
The proof is then divided into three subcases by symmetry.

\medskip
\noindent{\bf Case 2.1.} $w_2\in X_3$.
\medskip

By symmetry and Claim \ref{X_4_2}, if $d(f_u^1)=3$ for $u\in\{x,y,z\}$, then $u_2\in X_1$.
By Claim \ref{6+v}, $m_{4^+}(u)\ge2$ and $\tau(u\to T_v)\ge\frac12\times2=1$ for $u\in\{x,y\}$; and $m_{4^+}(u)\ge1$ and $\tau(u\to T_v)\ge\frac25\times2=\frac45$ for $u\in\{z,w\}$.
Now consider $w$.
If $d(w)\ge7$, or $d(w)=6$ with $m_{4^+}(w)\ge2$, then $\tau(w\to T_v)\ge\frac12\times2=1$ by Claim \ref{6+v}.
If $d(w)=6$ with $m_{4^+}(w)=1$,  we consider two possibilities.

$\bullet$ Suppose $d(f_x^1)\ge4$.
If $w_4$ is true, then $U\cup\{w_3,w_4\}$ is full and we set $C_6=xyzw_3w_4wx$.
Otherwise, $w_4$ is false.
If $d(f_x^1)=4$, i.e., $f_x^1=[xww_4x_2]$, then $U\cup\{w_3,x_2\}$ is full  and we set $C_6=xyzww_3x_2x$.
Hence $d(f_x^1)\ge5$, so $\tau(f_x^1\to T_v)\ge\frac15$ by (R2).
Therefore $\tau(w\to T_v)+\tau(f_x^1\to T_v)\ge\frac45+\frac15=1$.

$\bullet$ Suppose $d(f_x^1)=3$.
Let $t=d(f_w^3)$ and write $f_w^3=[ww_3s_1\cdots s_{t-3}x_2]$.
If $t=4$, then $f_w^3=[ww_3s_1x_2]$, and by (P1) $s_1$ is true.
Then $U\cup\{w_3,s_1\}$ is full and we set $C_6=xyzww_3s_1x$.
If $t=5$, write $f_w^3=[ww_3s_1s_2x_2]$.
By (P1), $s_2$ is true.
If $s_1$ is true, then $S=\{x,z,w,w_3,s_1,s_2\}$ is full and we set $C_6=xwzw_3s_1s_2x$.
If $s_1$ is false, let $s_2,w_3,q_1,q_2$ be the neighbors of $s_1$ in $H$.
If $w_3q_1\in E(G'')$, then $S=\{x,z,w,w_3,q_1,s_2\}$ is full and we set $C_6=xwzw_3q_1s_2x$; thus $w_3q_1\notin E(G'')$, so $m_{4^+}(f_w^3)\ge2$.
By (R2), $\tau(f_w^3\to T_{x_2})\ge\frac{5-4}{5-2}=\frac13$.
If $t\ge6$, then $\tau(f_w^3\to T_{x_2})\ge\frac{t-4}{t-1}\ge\frac25$ by (R2).
In all cases, we have $\tau(f_w^3\to T_{x_2})\ge\frac13$.
After (R0)--(R2) are carried out, $\tau(T_{x_2}\to T_v)=\sigma(T_{x_2})\ge-1+\frac12+\frac25+\frac13=\frac7{30}$.
Therefore $\tau(w\to T_v)+\tau(T_{x_2}\to T_v)\ge\frac45+\frac7{30}=\frac{31}{30}$.

The discussion for $z$ is analogous.
Consequently, $w'(T_v)\ge-4+1\times4=0$.

\medskip
\noindent{\bf Case 2.2.} $w_2\in X_2^a$.
\medskip

Assume, w.l.o.g., that $d(f_w^2)=3$. Then $w_3$ is true.
For $u\in\{x,y,z\}$, Claim \ref{X_4_2} gives $m_{4^+}(u)\ge2$, hence $\tau(u\to T_v)\ge\frac12\times2=1$ by Claim \ref{6+v}.
If $d(w)\ge7$, or $d(w)=6$ with $m_{4^+}(w)=2$, then $\tau(w\to T_v)\ge1$ by Claim \ref{6+v}, and hence $w'(T_v)\ge-4+1\times4=0$.
Thus we may assume $d(w)=6$ and $m_{4^+}(w)\le1$.
Suppose $d(f_x^1)\ge4$. Then $d(f_w^3)=3$.
As in Case 2.1, we obtain that $w_4$ is false and $d(f_x^1)\ge5$, so $\tau(f_x^1\to T_v)\ge\frac15$ by (R2).
Therefore $w'(T_v)\ge-4+1\times3+\frac45+\frac15=0$.
By Claim \ref{X_4_2}, $d(f_x^1)=3$ with $x_2\in X_1\cup X_2^a$.

If $x_2\in X_1$, then similarly to Case 2.1 we get $d(f_w^3)\ge5$ and $\tau(f_w^3\to T_{x_2})\ge\frac13$.
After (R0)--(R2) are carried out, by (r2) we have $\tau(T_{x_2}\to T_v)=\sigma(T_{x_2})\ge\frac7{30}$, hence $w'(T_v)\ge-4+1\times3+\frac45+\frac7{30}=\frac1{30}$.
If $x_2\in X_2^a$, then for some $u\in\{x,y,z\}$ with $m_{4^+}(u)\ge3$, Claim \ref{6+v} gives $\tau(u\to T_v)\ge\frac23\times2=\frac43$; then $w'(T_v)\ge-4+1\times2+\frac43+\frac23=0$.
Thus we may assume $m_{4^+}(u)=2$ for all $u\in\{x,y,z\}$.
We now discuss the following cases by symmetry.

$\bullet$ Suppose $d(f_y^1), d(f_z^1)\ge4$.
If $d(f_y^1)=4$, write $f_y^1=[yxx_{d(x)-2}y_2]$.
If $x_{d(x)-2}$ and $y_2$ are true, then $U\cup\{y_2,x_{d(x)-2}\}$ is full and we set $C_6=xwzyy_2x_{d(x)-2}x$.
If $x_{d(x)-2}$ is true and $y_2$ is false, then $y_3$ is true by (P1) and we set $C_6=xwzyy_3x_{d(x)-2}x$.
If $x_{d(x)-2}$ is false and $y_2$ is true, then $x_{d(x)-3}$ is true by (P1) and we set $C_6=xwzyy_2x_{d(x)-3}x$.
Thus, by symmetry, $d(f_y^1), d(f_z^1)\ge5$, so $\tau(f_y^1\to T_v)\ge\frac15$ and $\tau(f_z^1\to T_v)\ge\frac15$ by (R2).
Consequently $w'(T_v)\ge-4+1\times3+\frac23+\frac15\times2=\frac1{15}$.

$\bullet$ Suppose $d(f_y^1)=3$ and $d(f_z^1)\ge4$.
By Claim \ref{X_4_2}, $y_2\in X_1$.
As in the previous case, we obtain $d(f_z^1)\ge5$ and $\tau(f_z^1\to T_v)\ge\frac15$.
If $d(f_y^2)\ge5$, then by (R2) $\tau(f_y^2\to T_{y_2})\ge\frac15$, and after (R0)--(R2)  we have $\tau(T_{y_2}\to T_v)\ge\sigma(T_{y_2})\ge-1+\frac12\times2+\frac15=\frac15$.
Hence $w'(T_v)\ge-4+1\times3+\frac23+\frac15\times2=\frac1{15}$.
If $d(f_y^2)=4$, write $f_y^2=[yy_2qy_3]$.
By (P1), $q$ is true.
If $y_3$ is true, then $U\cup\{y_3,q\}$ is full and we set $C_6=xwzyy_3qx$; if $y_3$ is false, then $y_4$ is true and $U\cup\{y_4,q\}$ is full, so we set $C_6=xwzyy_4qx$.

$\bullet$ Suppose $d(f_y^1)=d(f_z^1)=3$.
If $d(y)\ge7$, then $\tau(y\to T_v)\ge\frac35\times2=\frac65$ by Claim \ref{6+v}.
 As before, we get $d(f_y^2)\ge5$ and $\tau(f_y^2\to T_{y_2})\ge\frac15$, so after (R0)--(R2) $\tau(T_{y_2}\to T_v)\ge\sigma(T_{y_2})\ge-1+\frac12+\frac35+\frac15=\frac3{10}$.
 Consequently $w'(T_v)\ge-4+1\times2+\frac65+\frac23+\frac3{10}=\frac16$.
Thus $d(y)=6$.
If $d(f_y^2)=d(f_y^3)=4$, then by the previous discussion $y_3$ is false.
Write $f_y^2=[yy_2qy_3]$ and $f_y^3=[yy_3pz_2]$; then $U\cup\{q,p\}$ is full and we set $C_6=xqpzwyx$.
Hence we may assume $d(f_y^2)\ge5$.
Since $m_{4^+}(f_y^2)\ge2$, $\tau(f_y^2\to T_{y_2})\ge\frac{d(f_y^2)-4}{d(f_y^2)-2}\ge\frac13$ by (R2).
After (R0)--(R2) are carried out, $\tau(T_{y_2}\to T_v)\ge\sigma(T_{y_2})\ge-1+\frac12\times2+\frac13=\frac13$ by (r1), so $w'(T_v)\ge-4+1\times3+\frac23+\frac13=0$.

\medskip
\noindent{\bf Case  2.3.} $w_2\in X_1$.
\medskip

From the previous subcases, we may assume that for each $u\in\{x,y,z\}$, either $d(f_u^1)\ge4$ or $d(f_u^1)=3$ with $u_2\in X_1$.
Hence $m_{4^+}(u)\ge2$ for all $u\in U$.
By Claim \ref{6+v}, this yields $\tau(u\to T_v)\ge1$ for all $u\in U$, and consequently $w'(T_v)\ge-4+1\times4=0$.

\end{document}